\documentclass[12pt]{article}
\usepackage{amssymb}
\usepackage{amsmath}
\usepackage{amsthm}
\usepackage{enumitem}
\usepackage{graphicx}

\usepackage[total={6.5in,8.75in}, top=1.2in, left=0.9in, includefoot]{geometry}

\usepackage{epstopdf}
\graphicspath{{./Figures/}{.}}

\def\cS{\mathcal{S}}
\def\mob{M\"{o}bius}
\def\ncss{$n$-cycle solution system}
\def\rot{rotation number}
\def\sew{symbol sequence}
\def\sL{{\sf L}}
\def\sR{{\sf R}}
\def\sw{switching manifold}
\def\tsp{terminating shrinking point}

\begin{document}

\newtheorem{theorem}{Theorem}
\newtheorem{corollary}[theorem]{Corollary}
\newtheorem{lemma}[theorem]{Lemma}
\theoremstyle{definition}
\newtheorem{definition}{Definition}

\title{Shrinking Point Bifurcations of Resonance
Tongues for Piecewise-Smooth, Continuous Maps}
\author{D.J.W.~Simpson and J.D.~Meiss\thanks{
D.~J.~W.~Simpson and J.~D.~Meiss gratefully
acknowledge support from NSF grant DMS-0707695.
}\\
Department of Applied Mathematics\\
University of Colorado\\
Boulder, CO 80309-0526}
\maketitle


\begin{abstract}
Resonance tongues are mode-locking regions of parameter space in which stable periodic solutions occur; they commonly occur, for example, near Neimark-Sacker bifurcations. For piecewise-smooth, continuous maps these tongues typically have a distinctive lens-chain (or sausage) shape in two-parameter bifurcation diagrams. We give a symbolic description of a class
of ``rotational" periodic solutions that display lens-chain structures for a general $N$-dimensional map. We then unfold the codimension-two, shrinking point bifurcation,
where the tongues have zero width. A number of codimension-one bifurcation curves emanate from shrinking points and we determine those that form tongue boundaries.
\end{abstract}

\section{Introduction}
\label{sec:INTRO}

Mathematical models often incorporate a nonsmooth component
in order to describe a physical discontinuity or sudden change.
Piecewise-smooth, continuous maps are a class of such systems.
We say that a map
\begin{equation}
x' = F(x) \;,
\label{eq:genMap}
\end{equation}
where $F : \mathbb{R}^N \to \mathbb{R}^N$,
is {\em piecewise-smooth continuous} if $F$ is everywhere continuous
and $\mathbb{R}^N$ can be partitioned into countably many regions where $F$ has
a different smooth functional form.
The map is nondifferentiable on codimension-one region
boundaries called {\em \sw s}.

Piecewise-smooth, continuous maps have been used as models in many areas,
for example, economics \cite{PuSu06,LaMo06},
power electronics \cite{ZhMo03,BaVe01,Ts03},
and cellular neural networks \cite{ChJu04}.
Furthermore, they arise as Poincar\'{e} maps of piecewise-smooth
systems of differential equations,
particularly near sliding bifurcations \cite{DiKo02,Ko05}
and near so-called corner collisions \cite{DiBu01c,OsDi08}.

Piecewise-smooth, continuous maps may exhibit
{\em border-collision bifurcations}
when a fixed point crosses a \sw~under
smooth variation of system parameters \cite{NuYo92,DiBu08,LeNi04,BaGr99}.
A fundamental characteristic of these bifurcations is that multipliers
associated with the fixed point may change discontinuously
at the bifurcation.
In some cases, border-collision bifurcations resemble bifurcations of smooth maps
(such as saddle-node and flip bifurcations); however, invariant sets created at the bifurcation generically exhibit a linear (not quadratic) growth in size
under variation of system parameters.
Alternatively, border-collision bifurcations may have no smooth analogue
(for example, an immediate transition from a fixed point to a
period-three cycle in a one-dimensional map).
Furthermore, these bifurcations may be extremely complicated,
producing multiple attractors and chaotic dynamics.

Resonance (or Arnold) tongues
are regions in parameter space within which there is an attracting periodic solution. In some cases (e.g., circle maps), the periodic solutions in a tongue can be characterized by their rotation number.
As was first observed for a one-dimensional, piecewise-linear,
circle map \cite{YaHa87},
two-parameter pictures of resonance tongues for piecewise-smooth,
continuous maps typically exhibit a distinctive lens-chain
(or sausage) geometry
\cite{SiMe08b,SuGa04,PuSu06,ZhMo06b,ZhMo08b},
as in Fig.~\ref{fig:res}.
The boundaries of a resonance tongue correspond either to a loss of stability
of the associated periodic solution or
to a collision of one point of the solution with the \sw.
The latter case usually corresponds to the collision of the stable periodic
solution with an unstable periodic solution of the same period
and is known as a border-collision fold bifurcation.
(This need not always occur, as is shown in 
\cite{SiMe08b}.)
Following \cite{YaHa87}, we call the codimension-two points
at which resonance tongues have zero width, {\em shrinking points}
(sometimes called ``waist points'').
The purpose of the present paper is to
explore the bifurcation that occurs near such points for a general $N$-dimensional map
in the neighborhood of a \sw.

\begin{figure}[ht]
\begin{center}
\setlength{\unitlength}{1cm}
{\includegraphics[width=15cm,height=8.4cm]{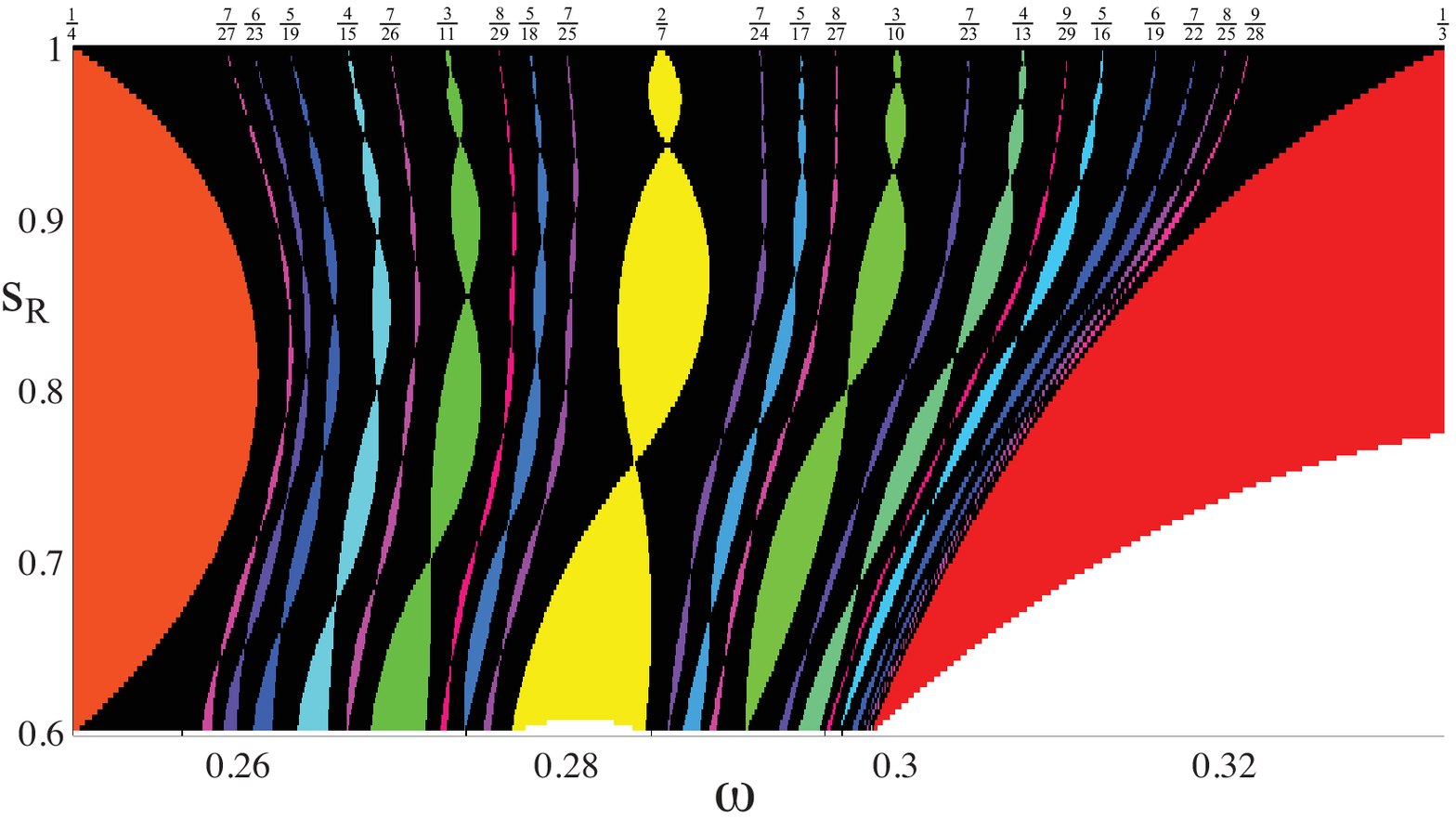}}      
\end{center}
{Figure 1: Resonance tongues corresponding to period-$n$ cycles
up to $n = 30$ of (\ref{eq:pwaMap}) with
$A_\sL = \begin{bmatrix}
			\frac65 \cos(2 \pi \omega) & 1 \\
			-\frac{9}{25} & 0
		 \end{bmatrix}$, 
$A_\sR =\protect{\begin{bmatrix}
			\frac{2}{s_\sR} \cos(2 \pi \omega) & 1 \\
			-\frac{1}{s_\sR^2} & 0
		\end{bmatrix}}$, 
$b = e_1$ and $\mu = 1$, as shown in \cite{SiMe08b}.
The rotation number associated with each resonance tongue is indicated.}
\label{fig:res}
\end{figure}
\addtocounter{figure}{1}

A \sw~locally divides phase space into two regions. Any nearby period-$n$ orbit with no points on the \sw~consists of, say, $l$ points on one side and $n-l$ points on the other side of the \sw. We observe that $l$ differs by one for coexisting stable and unstable periodic orbits in lens shaped resonance regions.  Moreover, the number $l$ also differs by one for stable periodic orbits in adjoining lenses. For example, Fig.~\ref{fig:shrink27set} shows a magnification of Fig.~\ref{fig:res} near a generic shrinking point in the $2/7$-tongue. The borders of the lens near such a shrinking point correspond to border-collision bifurcations of the two orbits in the lens, and thus to the occurrence of an orbit with a point on the \sw.

Consequently, a shrinking point is defined geometrically by the occurrence of a periodic solution that two points on the \sw. However, in \S\ref{sec:SHRINK} we show that there are two cases that must be treated separately: terminating and non-terminating shrinking points. Non-terminating shrinking points are defined algebraically, but by the singularity of certain matrices (the `border collision" matrices). Terminating shrinking points correspond to the case $l = 1$ or $n-1$; they are also defined algebraically by the occurrence of certain eigenvalues.  We will prove in  Lem.~\ref{le:shrinkPoly} (under certain nondegeneracy conditions, see  Defs.~\ref{def:shrink} and \ref{def:tershr}) the following.
\begin{itemize}
	\item The algebraic shrinking point conditions imply
	the geometrical condition of having two points on the \sw.
	\item At a shrinking point, the map has an invariant polygon 
	(that is typically nonplanar) on which the dynamics are conjugate 
	to a rigid rotation.
\end{itemize}

Since a shrinking point is defined by two conditions (two points on the \sw), the unfolding
of a shrinking point bifurcation requires the variation of two parameters. In the generic two-parameter picture, a (non-terminating) shrinking point appears, at first glance, to lie at the intersection of two smooth curves. However, this is an illusion: there are really four distinct curves corresponding to four different border-collision fold bifurcations. We will prove in Thm.~\ref{th:shrink} the following.
\begin{itemize}
	\item Each of the four border-collision curves is quadratically 
	tangent to one of the others and appears on only one side 
	of the shrinking point.
	\item The extension of these curves through the shrinking point 
	corresponds to virtual solutions. 
	\item At each border-collision fold bifurcation one particular 
	point of a particular period-$n$ cycle lies on the \sw. 
	\item There are curves in parameter space along which each point of this cycle
	lies on the \sw, see Fig.~\ref{fig:shrink27set}; however except for the four that form the resonance tongue boundaries, these correspond to virtual solutions. 

\end{itemize}


\begin{figure}[ht!]
\begin{center}
\setlength{\unitlength}{1cm}
\begin{picture}(12,10)
\put(0,0){\includegraphics[width=12cm,height=10cm]{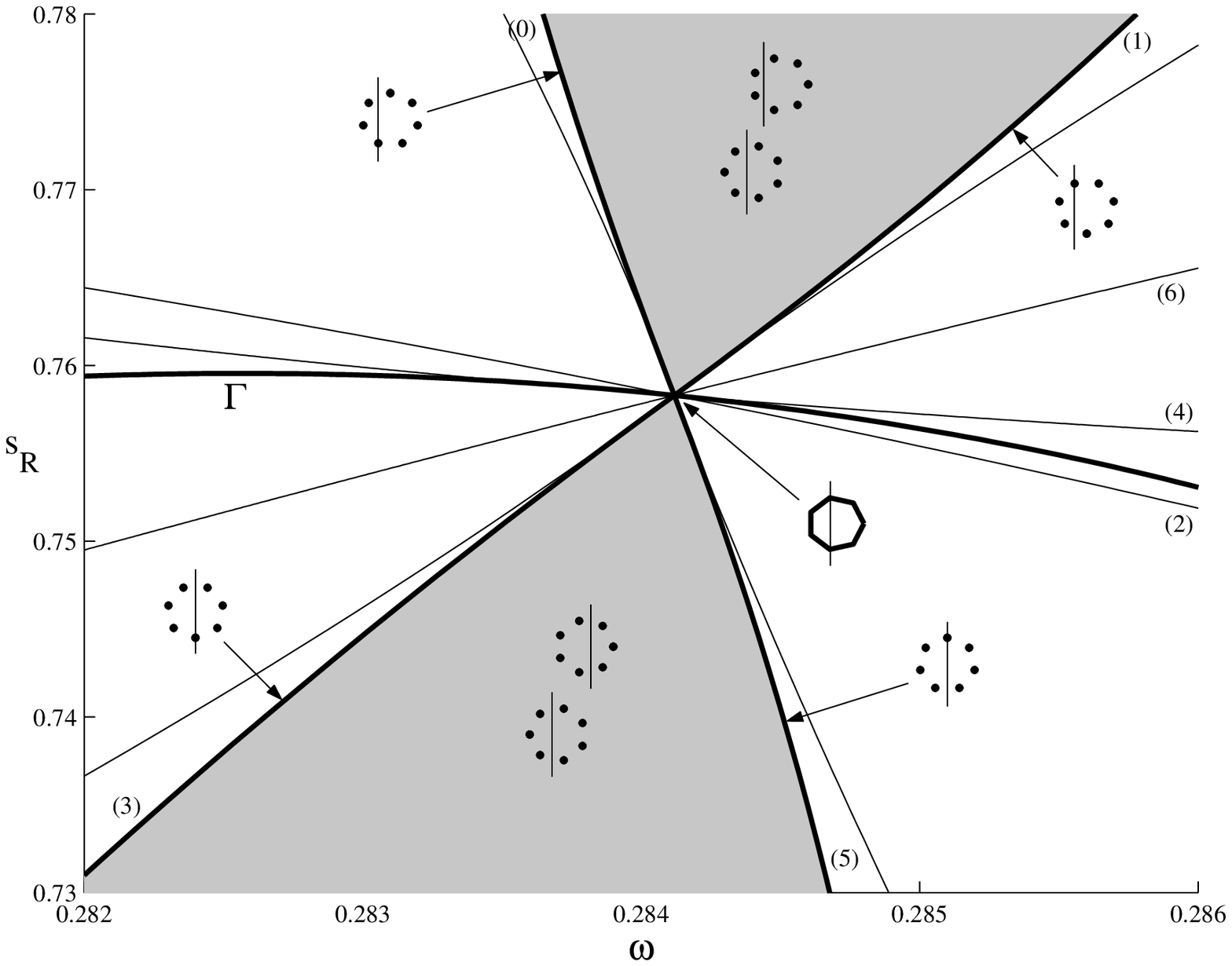}}
\end{picture}
\end{center}
\caption{\small{
A magnification of Fig.~\ref{fig:res} showing
the 2/7-resonance tongue near a shrinking point.
Sketches of four period-$7$ orbits are shown. The orbit with $l=3$ 
(its symbol sequence is ${\sf LLRRLRR}$, see \S\ref{sec:SYMDYN})
is admissible in both gray regions.  
In the upper region this orbit is stable; in the lower region it is a saddle.
Let $\{ x_i \}$ denote the points on this orbit. 
On the curves labeled $(i)$ the point $x_i$ lies on the \sw.
By Th.~\ref{th:shrink}, the curves $(0)$, $(1)$, $(3)$ and $(5)$ form tongue boundaries.
A saddle period-$7$ orbit with $l=2$ is
admissible in the upper region and coincides with $\{ x_i \}$ on
the boundaries (0) and (1). Similarly, a stable period-$7$ orbit with $l = 4$ is
admissible in the lower region and coincides with $\{ x_i \}$ on
the boundaries (3) and (5). At the shrinking point there is an invariant heptagon.
On the curve $\Gamma$ the linear solution system 
(\ref{eq:nCycleSolutionSystem}) is singular.}}
\label{fig:shrink27set}
\end{figure}

The remainder of this paper is organized as follows.
In \S\ref{sec:FRAME} we introduce the general map
that will be studied throughout this paper and 
determine fixed points.
Symbolic dynamics are introduced in \S\ref{sec:SYMDYN}
to describe periodic solutions.
The periodic orbit corresponding to a particular symbol sequence
may be found by solving a linear system, see \S\ref{sec:DESCR}.
In \S\ref{sec:ROTAT} we define the new concept of a
``rotational symbol sequence''
and using symbolic dynamics are able to classify 
periodic solutions as either rotational or non-rotational.
The motivation is that rotational periodic solutions
correspond to lens-chain resonance regions
whereas non-rotational periodic solutions may not \cite{SiMe08b}.
In \S\ref{sec:CARD} we discuss the number of distinct
rotational periodic solutions.
Formal definitions for terminating and non-terminating shrinking points
are provided in \S\ref{sec:SHRINK}.
This section describes the singular nature of shrinking points
and the construction of invariant, nonplanar, polygons.
This mostly preparatory work sets up
\S\ref{sec:UNFOLD} that gives the unfolding of the  shrinking point bifurcation
summarized by Thm.~\ref{th:shrink}.
Concluding comments are given in \S\ref{sec:CONC}.

\section{A framework for local analysis}
\label{sec:FRAME}

It is well-known (see for instance \cite{DiBu08})
that dynamical behavior of the piecewise-smooth, continuous map (\ref{eq:genMap})
local to a nondegenerate border-collision bifurcation of a fixed point
at a smooth point on a \sw~is described by a piecewise-affine map of the form
\begin{equation}
x' = f_\mu(x) =
\left\{ \begin{array}{lc}
f_\mu^{\sL}(x), & s \le 0 \\
f_\mu^{\sR}(x), & s \ge 0
\end{array} \right. \;,
\label{eq:pwaMap}
\end{equation}
where
\begin{equation}
f_\mu^{i}(x) = \mu b + A_i x \;,
\label{eq:fmui}
\end{equation}
for $i = \sL,\sR$.
Here
\begin{equation}
s = e_1^{\sf T} x \;,
\end{equation}
denotes the first component of $x \in \mathbb{R}^N$,
$\mu \in \mathbb{R}$ is a system parameter,
$b \in \mathbb{R}^N$,
and $A_\sL$ and $A_\sR$ are $N \times N$ matrices,
which, by continuity of (\ref{eq:pwaMap}),
must be identical in all but possibly their first columns.
Nonlinear terms only affect local dynamics
in degenerate cases (for example when $A_\sL$ has a multiplier
on the unit circle).
Also, for simplicity, we have assumed the \sw~is analytic
and therefore may be transformed to the plane, $s = 0$.

The map (\ref{eq:pwaMap})
is a homeomorphism if and only if $\det(A_\sL) \det(A_\sR) > 0$.
Furthermore, (\ref{eq:pwaMap}) has the scaling symmetry
\begin{equation}
f_{\lambda \mu}(\lambda x) = \lambda f_\mu(x),~~\forall \lambda > 0 \;,
\label{eq:muScaling}
\end{equation}
and consequently every bounded invariant set of (\ref{eq:pwaMap})
collapses to the origin as $\mu \to 0$.
Furthermore, it suffices to consider only $\mu = -1,0,1$.

Potential fixed points of (\ref{eq:pwaMap}) are obtained by solving
$x^{*(i)} = f_\mu^i(x^{*(i)})$.
Whenever $A_i$ does not have a multiplier of one,
$x^{*(i)}$ is unique and given by
\begin{equation}
x^{*(i)} = \mu (I - A_i)^{-1} b \;.
\label{eq:xStari}
\end{equation}
Each $x^{*(i)}$, as given by (\ref{eq:xStari}),
is a fixed point of (\ref{eq:pwaMap}) and said to be {\em admissible}
whenever $s^{*(i)}$ has the appropriate sign
(negative for $i = \sL$ and positive for $i = \sR$),
otherwise it is called {\em virtual}.
We have
\begin{equation}
s^{*(i)} = \frac{\mu}{\det(I-A_i)} e_1^{\sf T} {\rm adj}(I-A_i) b \;,
\nonumber
\end{equation}
where ${\rm adj}(X)$ denotes the {\em adjugate matrix} of a matrix $X$,
(${\rm adj}(X) X = \det(X) I$).
Since $A_\sL$ and $A_\sR$ differ in only possibly the first column,
their adjugates share the same first row,
which we denote by $\varrho^{\sf T}$:
\begin{equation}
\varrho^{\sf T} = e_1^{\sf T} {\rm adj}(I-A_\sL)
= e_1^{\sf T} {\rm adj}(I-A_\sR) \;.
\label{eq:varrho}
\end{equation}
Thus
\begin{equation}
s^{*(i)} = \frac{\mu \varrho^{\sf T} b}{\det(I-A_i)} \;.
\label{eq:xStari1}
\end{equation}
The condition $\varrho^{\sf T} b \ne 0$ is a nondegeneracy
condition for the border-collision bifurcation,
guaranteeing that the distance between each $x^{*(i)}$
and the \sw~varies linearly with the bifurcation parameter $\mu$.
Furthermore, if $(I-A_i)$ is nonsingular and
$a_i$ denotes the number of real multipliers of $A_i$
that are greater than one, then $\det(I-A_i)$ is positive
if and only if $a_i$ is even.
Thus, as described in \cite{DiFe99},
by (\ref{eq:xStari1}),
if $a_L + a_R$ is even, $x^{*(\sL)}$ and $x^{*(\sR)}$
are admissible for different signs of $\mu$,
thus as $\mu$ is varied through zero a single fixed point {\em persists}.
Conversely if $a_L + a_R$ is odd,
$x^{*(\sL)}$ and $x^{*(\sR)}$ coexist for the same sign of $\mu$.
At $\mu = 0$ the fixed points collide and annihilate in a {\em nonsmooth fold}
bifurcation.

\section{Symbolic dynamics}
\label{sec:SYMDYN}

We refer to any sequence, $\cS$, that has elements taken
from the alphabet, $\{ \sL, \sR \}$, as a \sew.
This paper focuses on periodic solutions; thus,
we will assume $\cS$ is finite and let $n$ denote the length of $\cS$.
We index the elements of $\cS$ from $0$ to $n-1$.
Arithmetic on the indices of $\cS$ is usually modulo $n$.
For clarity, throughout this paper we omit ``${\rm mod \,} n$'' where
it is clear modulo arithmetic is being used.

Here we introduce notation relating to \sew s.
Given $n \in \mathbb{N}$, the collection of all symbol sequences of
length $n$ is $\{ \sL, \sR \}^n \equiv \Sigma_2^n$.
The $i^{\rm th}$ left cyclic permutation is an operator,
$\sigma_i : \Sigma_2^n \to \Sigma_2^n$,
defined by $(\sigma_i \cS)_j = \cS_{i+j}$
and we use the notation, $\cS^{(i)} \equiv \sigma_i \cS$.
The $i^{\rm th}$ flip permutation is an operator,
$\chi_i : \Sigma_2^n \to \Sigma_2^n$,
that flips the $i^{\rm th}$ element of $\cS$
(i.e.~$\sL \to \sR$ and $\sR \to \sL$)
and leaves all other elements unchanged.
We use the notation, $\cS^{\overline{i}} \equiv \chi_i \cS$.
For example if
$\cS = {\sf LRLRR}$ then
$\cS^{\overline{3}} = {\sf LRLLR}$ and
$\cS^{(2)} = {\sf LRRLR}$.
In general
$\cS^{\overline{i}(j)} \equiv (\cS^{\overline{i}})^{(j)}
\ne (\cS^{(j)})^{\overline{i}} \equiv \cS^{(j)\overline{i}}$
because with the same example
$\cS^{\overline{3}(2)} = {\sf LLRLR}$ and
$\cS^{(2)\overline{3}} = {\sf LRRRR}$.

The $i^{\rm th}$ multiplication permutation is an operator,
$\pi_i : \Sigma_2^n \to \Sigma_2^n$,
defined by $(\pi_i \cS)_j = \cS_{ij}$.
Notice $(\pi_i \pi_j \cS)_k = (\pi_j \cS)_{ik}
= \cS_{ijk} = (\pi_{ij} \cS)_k$ thus
\begin{equation}
\pi_i \pi_j = \pi_{ij} \;.
\label{eq:permMultiply}
\end{equation}
Consequently $\pi_i$ is an invertible operator if and only if
${\rm gcd}(i,n) = 1$ and the inverse of $\pi_i$ is $\pi_{i^{-1}}$
(where $i^{-1}$ is the multiplicative inverse of $i$ modulo $n$).
Also notice $(\sigma_i \pi_j \cS)_k = (\pi_j \cS)_{i+k}
= \cS_{ij+jk} = (\sigma_{ij} \cS)_{jk} = (\pi_j \sigma_{ij} \cS)_k$.
Hence
\begin{equation}
\sigma_i \pi_j = \pi_j \sigma_{ij} \;.
\label{eq:permReverse}
\end{equation}

We let $\cS \hat{\cS}$ denote the concatenation of the \sew s
$\cS$ and $\hat{\cS}$ and let $\cS^k \in \Sigma_2^{kn}$
denote the \sew~formed by the concatenation of $k$ copies of $\cS$.
A \sew~is called {\em primitive} if it cannot be written
as a power, $\cS^k$, for any $k > 1$.
$\cS$ is primitive if and only if $\cS \ne \cS^{(i)}$ for all $i \ne 0$.

\section{Describing and locating periodic solutions}
\label{sec:DESCR}

Each orbit of (\ref{eq:pwaMap}) can be coded by a \sew~that gives
its itinerary relative to the \sw.
However, instead of defining symbol sequences for orbits,
we find it preferable to do the reverse.
Given a point $x = x_0 \in \mathbb{R}^N$,
we denote $x_i$ as the $i^{\rm th}$ iterate of $x$
under the maps $f_\mu^{\sL}$ and $f_\mu^{\sR}$ in the order determined by
$\cS \in \Sigma_2^n$:
\begin{equation}
x_{i+1} = f_\mu^{\cS_i}(x_i)
\label{eq:altMap}
\end{equation}
In general this is different from iterating $x$ under the map (\ref{eq:pwaMap}).
However, if the sequence $\{ x_i \}$
satisfies the {\em admissibility condition}:
\begin{equation}
\cS_i = \left\{
\begin{array}{lc}
\sL, & {\rm whenever~} s_i < 0 \\
\sR, & {\rm whenever~} s_i > 0
\end{array}
\right.
\label{eq:admissCond}
\end{equation}
for every $i$, then $\{ x_i \}$
coincides with the forward orbit of $x$ under (\ref{eq:pwaMap}).
When (\ref{eq:admissCond}) holds for every $i$,
$\{ x_i \}$ is admissible,
otherwise it is virtual.

For a given \sew~$\cS$,
we are interested in finding $x_0 \in \mathbb{R}^N$
such that $x_0 = x_n$,
because then $\{ x_0, x_1, \ldots, x_{n-1} \}$
is a period-$n$ orbit.
We call this orbit an {\em $\cS$-cycle}.
$\cS$-cycles are determined by the linear system
\begin{eqnarray*}
x_1 & = & A_{\cS_0} x_0 + \mu b \;, \\
x_2 & = & A_{\cS_1} x_1 + \mu b \;, \\
& \vdots & \\
x_0 & = & A_{\cS_{n-1}} x_{n-1} + \mu b  \;.
\end{eqnarray*}
Elimination of the points $x_1,\ldots,x_{n-1}$, gives
\begin{equation}
(I - M_\cS) x_0 = \mu P_\cS b \;.
\label{eq:nCycleSolutionSystem}
\end{equation}
where
\begin{eqnarray}
M_\cS & = & A_{\cS_{n-1}} \ldots A_{\cS_0} \;, \label{eq:stabMatrix} \\
P_\cS & = &
I + A_{\cS_{n-1}} + A_{\cS_{n-1}} A_{\cS_{n-2}} + \cdots
+ A_{\cS_{n-1}} \ldots A_{\cS_1}  \;. \label{eq:bcMatrix}
\end{eqnarray}
We call (\ref{eq:nCycleSolutionSystem}) the {\em \ncss} of $\mathcal{S}$.
If ($I -M_\cS$) is nonsingular, then
(\ref{eq:nCycleSolutionSystem}) has the unique solution
\begin{equation}
x_0 = \mu (I - M_\cS)^{-1} P_\cS b \;.
\label{eq:nCycleSolution}
\end{equation}
Stability of the period-$n$ orbit is determined by $M_\cS$
and for this reason we call $M_\cS$ the {\em stability matrix} of $\cS$.
In view of Lem.~\ref{le:bc} (see below),
we call $P_\cS$ the {\em border-collision matrix} of $\cS$.
Notice $P_\cS$ is independent of $\cS_0$, thus
\begin{equation}
P_\cS = P_{\cS^{\overline{0}}} \;.
\label{eq:PindepS0}
\end{equation}
Also, it is easily verified that $M_\cS$ and $M_{\cS^{\overline{0}}}$
differ in only their first column.
Consequently, the map that describes the $n^{\rm th}$ iterate of $x_0$
under either $\cS$ or $\cS^{\overline{0}}$:
\begin{equation}
x_n = \left\{ \begin{array}{lc}
\mu P_\cS b + M_\cS x, & s \le 0 \\
\mu P_\cS b + M_{\cS^{\overline{0}}} x, & s \ge 0
\end{array} \right. \;,
\label{eq:nthPWA}
\end{equation}
is piecewise-smooth continuous
and has the same form as (\ref{eq:pwaMap}).
For this reason (\ref{eq:pwaMap}) may be used to investigate
dynamical behavior local to border-collision bifurcations of
periodic solutions.

We now state five fundamental lemmas relating to \ncss s
that we will utilize in \S\ref{sec:SHRINK} and \S\ref{sec:UNFOLD}.
Lems.~\ref{le:bc} and \ref{le:sn}
are generalizations of those given in \cite{SiMe08b}.

\begin{lemma} 
Suppose $x$ solves the \ncss~(\ref{eq:nCycleSolutionSystem})
of $\cS$ and $s_i = e_1^{\sf T} x_i = 0$.
Then $x$ also solves the \ncss~of $\cS^{\overline{i}}$.
\label{le:solvesAlso}
\end{lemma}

\begin{proof}
By continuity:
$A_\sL x_i = A_\sR x_i$, hence there is no restriction on
the $i^{\rm th}$ element of $\cS$.
\hfill
\end{proof}

\begin{lemma} 
Suppose $x$ and $\hat{x}$ solve the \ncss s of
$\cS$ and $\cS^{\overline{0}}$ respectively.
Then $\det(I-M_\cS) s = \det(I-M_{\cS^{\overline{0}}}) \hat{s}$.
\label{le:detPair}
\end{lemma}

\begin{proof}
By (\ref{eq:nCycleSolutionSystem}) and (\ref{eq:PindepS0}), we have
$(I-M_\cS) x = (I-M_{\cS^{\overline{0}}}) \hat{x}$.
Since $(I-M_\cS)$ and $(I-M_{\cS^{\overline{0}}})$
are identical except in the first column,
the first row of their adjugates are identical.
Multiplication on the left by this row to the previous equation
yields the desired result.
\hfill
\end{proof}

\begin{lemma} 
For any $i$, $\det(I-M_{\cS^{(i)}}) = \det(I-M_\cS)$.
\label{le:cyclicDetM}
\end{lemma}

\begin{proof}
Suppose w.l.o.g., $\mathcal{S}_0 = \sL$.
If $A_\sL$ is nonsingular, then
$(I-M_{\cS^{(1)}}) = A_\sL (I-M_\cS) A_\sL^{-1}$
which verifies the result for $i = 1$.
Since nonsingular matrices are dense in the set of all matrices
and the determinant of a matrix is a continuous function of its elements,
the result for $i = 1$ is also true even when $A_\sL$ is singular.
Repetition of this argument completes the result for any $i$.
\hfill
\end{proof}

\begin{lemma} 
Suppose $(I-M_\cS)$ is nonsingular, $\mu \ne 0$ and $\varrho^{\sf T} b \ne 0$.
Then the point $x_0$,
given by (\ref{eq:nCycleSolution}),
lies on the switching manifold if and only if $P_\cS$ is singular.
\label{le:bc}
\end{lemma}

\begin{lemma} 
Suppose $P_\cS$ is nonsingular, $\mu \ne 0$ and $\varrho^{\sf T} b \ne 0$.
Then the \ncss~(\ref{eq:nCycleSolutionSystem})
has a solution if and only if $(I-M_\cS)$ is nonsingular.
\label{le:sn}
\end{lemma}

Proofs for Lems.~\ref{le:bc} and \ref{le:sn}
are given in appendix \ref{sec:PROOFSa}.
An interpretation of these last two lemmas is presented in Fig.~\ref{fig:MPgrid}.
The situation $\det(I-M_\cS) = \det(P_\cS) = 0$ is generically codimension-two.
When appropriate nondegeneracy conditions are satisfied it is
equivalent to a shrinking point, see \S\ref{sec:SHRINK}.

\begin{table}[tbp]
\begin{center}
\begin{tabular}{c|c|c|}
			& $\det(I-M_\cS) \ne 0$ & $\det(I-M_\cS) = 0$ \\
\hline
			& unique solution  &  \\
\raisebox{.75em}[0pt]{$\det(P_\cS) \ne 0$}
			& and $s_0 \ne 0$  & \raisebox{.75em}[0pt]{no solution}\\
\hline
			& unique solution  & possibly uncountably  \\
\raisebox{.75em}[0pt]{$\det(P_\cS) = 0$}
			& and $s_0 = 0$    & many solutions \\
\hline
\end{tabular}
\caption{A grid summarizing the nature of solutions to (\ref{eq:nCycleSolutionSystem})
when $\mu \ne 0$ and $\varrho^{\sf T} b \ne 0$
as determined by Lems.~\ref{le:bc} and \ref{le:sn}.
$M_\cS$ is the stability matrix of $\cS$, (\ref{eq:stabMatrix}),
and $P_\cS$ is the border-collision matrix of $\cS$, (\ref{eq:bcMatrix}).}
\label{fig:MPgrid}
\end{center}
\end{table}

\section{Rotational symbol sequences}
\label{sec:ROTAT}

Of particular interest is the situation that the map (\ref{eq:pwaMap})
exhibits an invariant, topological circle
that crosses the switching manifold at two points.
When the restriction of the map to this circle is homeomorphic to
a monotone increasing circle map we find that
periodic solutions on the circle can only have certain \sew s.
We call these sequences {\em rotational \sew s}.

\begin{definition}[Rotational Symbol Sequence]
Let $l,m,n \in \mathbb{N}$, with $l,m < n$ and ${\rm gcd}(m,n) = 1$.
Let $\cS = \cS[l,m,n]$ be the \sew~of length $n$ defined by
\begin{equation}
\cS_{id} = \left\{ \begin{array}{lc} \sL, & i = 0,\ldots,l-1 \\
\sR, & i = l,\ldots,n-1
\end{array} \right.
\label{eq:rotSSdef}
\end{equation}
where $d$ is the multiplicative inverse of $m$ modulo $n$,
i.e.~$dm = 1 {\rm ~mod~} n$.
Then we say $\cS$ and any cyclic permutation of $\cS$ is a
{\em rotational \sew}.
\end{definition}

Notice $d$ always exists and is unique because
$m/n$ is an irreducible fraction \cite{Ga98}.
For example if $(l,m,n) = (3,2,7)$ then $d = 4$,
hence $\cS_0 = \cS_4 = \cS_{8 {\rm \; mod \;} 7=1} = \sL$,
thus $\cS[3,2,7] = {\sf LLRRLRR}$.

A pictorial method for computing $\cS$ in terms of $l$, $m$ and $n$
is to select $n$ points on a circle, then draw a vertical line
through the circle
such that $l$ points lie to the left of the line,
see Fig.~\ref{fig:rotSymSeq}.
Label the first point to the left of the lower intersection
of the circle and line, point 0.
Move $m$ points clockwise from 0 and label this point 1.
Continue stepping clockwise labeling every $m^{\rm th}$ point
with a number that is one greater than the previous number
until all points are labeled.
Then $\cS_i = \sL$ if the point $i$
lies to the left of the vertical line and $\sR$ otherwise.

\begin{figure}[ht]
\begin{center}
\includegraphics[width=5cm,height=6cm]{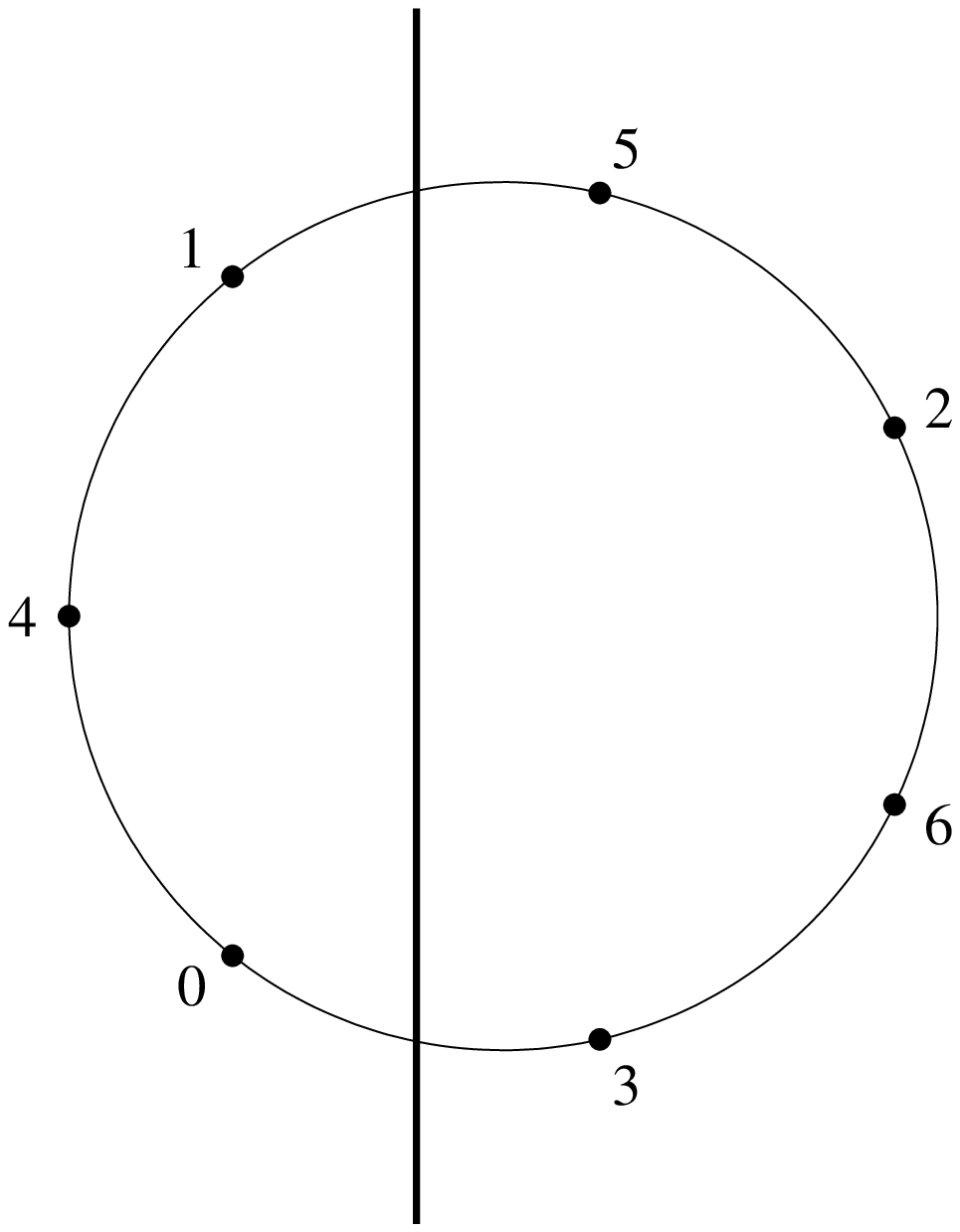}
\caption{
Illustration of the pictorial method for
determining $\cS[3,2,7] = {\sf LLRRLRR}$.
\label{fig:rotSymSeq}
}
\end{center}
\end{figure}

The circle represents an invariant circle of (\ref{eq:pwaMap})
and the vertical line represents the \sw.
Taking $m$ steps clockwise corresponds to
evaluating the map (\ref{eq:pwaMap}) once.
Thus $\cS[l,m,n]$ is the \sew~of an $n$-cycle of (\ref{eq:pwaMap})
that has $l$ points to the left of the \sw~and with {\em \rot}~$m/n$.

The following lemma states some basic properties of
rotational \sew s.

\begin{lemma}
Suppose $\cS[l,m,n]$ is a rotational symbol sequence.

\begin{enumerate}[label=\alph{*}),ref=\alph{*}]
\item
\label{it:constructS}
$\cS[l,m,n] = \pi_m \cS[l,1,n]$.
\item
\label{it:permDiffS}
If $\cS[l,m,n]^{(i)}_0 = \sL$ and $\cS[l,m,n]^{(i)}_{-d} = \sR$ then $i = 0$.
\item
\label{it:primitiveS}
$\cS[l,m,n]$ is primitive.
\item
\label{it:cycPerS}
$\cS[l,m,n]^{((l-1)d)} = \cS[l,n-m,n]$.
\item
\label{it:switchS}
$\cS[l,m,n]^{((l-1)d)\overline{0}} = \cS[l,m,n]^{\overline{0}(ld)}$.
\item
\label{it:notCycPerS}
If $0 < m_1, m_2 < \frac{n}{2}$ are distinct integers coprime to $n$
and $l \ne 1,n-1$,
then $\cS[l,m_1,n]$ is not a cyclic permutation of $\cS[l,m_2,n]$.
\end{enumerate}
\label{le:oss}
\end{lemma}

\begin{proof} 
Recall the notation $\sigma_i \cS = \cS^{(i)}$, for the $i^{th}$ left cyclic permutation and $\chi_i \cS = \cS^{\bar i}$ for the flip of the $i^{th}$ element.

\begin{enumerate}[label=\alph{*}),ref=\alpha{*}]
\item
Since $dm = 1$,
\begin{equation}
(\pi_m \cS[l,1,n])_{id} = \cS[l,1,n]_i = \sL
{\rm ~~if~and~only~if~~}i = 0,\ldots,l-1 \;,
\nonumber
\end{equation}
matching the definition of $\cS[l,m,n]$, (\ref{eq:rotSSdef}).
\item
Let $j = im$ (equivalently, $i = jd$).
Then $\cS[l,m,n]^{(i)}_0 = \cS[l,m,n]_{jd} = \sL$
implies $0 \le j \le l-1$.
Similarly, $\cS[l,m,n]^{(i)}_{-d} = \cS[l,m,n]_{(j-1)d} = \sR$
implies $l \le j-1 \le n-1$.
The only value of $j$ that satisfies both inequalities is zero, hence $i = 0$.
\item
From part (b), $\cS[l,m,n]$ differs from any nontrivial cyclic permutation of itself
in either the $0^{\rm th}$ or the $(-d)^{\rm th}$ element.
Therefore, $\cS[l,m,n]$ is primitive.
\item
By definition,
\begin{eqnarray}
\cS[l,m,n]_{id} = \sL & {\rm if~and~only~if} & i = 0,\ldots,l-1 \;, \label{eq:cSb1} \\
{\rm and~~~} \cS[l,n-m,n]_{i(n-d)} = \sL & {\rm if~and~only~if} & i = 0,\ldots,l-1 \;, \label{eq:cSb2}
\end{eqnarray}
since the multiplicative inverse of $(n-m)$ is $(n-d)$.
From (\ref{eq:cSb1}), by letting $j = i - l + 1$ we obtain
\begin{eqnarray}
\cS[l,m,n]_{jd}^{((l-1)d)} = \sL &
{\rm if~and~only~if} & j = -l+1,\ldots,0 \;, \label{eq:cSb3} \\
{\rm thus~~~} \cS[l,m,n]_{j(n-d)}^{((l-1)d)} = \sL &
{\rm if~and~only~if} & j = 0,\ldots,l-1 \;, \nonumber
\end{eqnarray}
which matches (\ref{eq:cSb2}).
\item
From (\ref{eq:cSb3}) we obtain
\begin{equation}
\cS[l,m,n]_{id}^{((l-1)d)\overline{0}} = \sL
{\rm ~~~if~and~only~if~~~} i = -l+1,\ldots,-1 \;. \label{eq:cSb4}
\end{equation}
Also
\begin{eqnarray}
\cS[l,m,n]_{id}^{\overline{0}} = \sL
& {\rm if~and~only~if} & i = 1,\ldots,l-1 \;, \nonumber \\
{\rm thus~~~} \cS[l,m,n]_{jd}^{\overline{0}(ld)} = \sL &
{\rm if~and~only~if} & j = -l+1,\ldots,-1 \;,
\label{eq:cSb5}
\end{eqnarray}
where we have set $j = i-l$. Equation (\ref{eq:cSb5}) matches (\ref{eq:cSb4})
proving the result.
\item
Let $d_1$ and $d_2$ denote the multiplicative inverses
of $m_1$ and $m_2$ modulo $n$, respectively.
Let
\begin{eqnarray}
\check{\cS} & = & \pi_{d_1} \cS[l,m_1,n] \;, \\
\hat{\cS} & = & \pi_{d_1} (\cS[l,m_2,n]^{(k)}) \;,
\end{eqnarray}
where $k \in \mathbb{Z}$.
Since $\pi_{d_1}$ is an invertible operator,
it remains to show that $\check{\cS} \ne \hat{\cS}$
for any $k \in \mathbb{Z}$.

Using (\ref{eq:permMultiply}) and part (a) we find
\begin{equation}
\check{\cS} = \cS[l,1,n] \label{eq:cScheck2} \;.
\end{equation}
Re-expressing $\hat{\cS}$ in the form we desire
is a little more complicated but requires
no more that the basic known properties concerning
multiplicative permutations, $\pi$.
\begin{eqnarray}
\hat{\cS} & = & (\pi_{d_1} \cS[l,m_2,n])^{(k m_1)} \;,
{\rm ~~by~} (\ref{eq:permReverse}) \nonumber \\
& = & (\pi_{d_1,n} \pi_{m_2,n} \cS[l,1,n])^{(k m_1)} \;,
{\rm ~~by~part~(a)} \nonumber \\
& = & (\pi_{\hat{m},n} \cS[l,1,n])^{(k m_1)} \;,
{\rm ~~where~} \hat{m} = d_1 m_2 \;,
{\rm ~by~} (\ref{eq:permMultiply}) \nonumber \\
& = & \cS[l,\hat{m},n]^{(k m_1)} \;,
{\rm ~~by~part~(a)} \label{eq:cShat2}.
\end{eqnarray}
Notice $\hat{m} \ne 1$ since $m_1 \ne m_2$.
Also $\hat{m} \ne n-1$ since, otherwise,
$d_1 m_2 = -1 \Rightarrow (n-d_1) m_2 = 1 \Rightarrow m_2 = n-m_1$
(since the multiplicative inverse of $(n-d_1)$ is $(n-m_1)$),
thus $m_1 + m_2 = n$ which is a contradiction
since $m_1,m_2 < \frac{n}{2}$ by assumption.

Using (\ref{eq:cShat2}), by the definition of an rotational \sew~(\ref{eq:rotSSdef}),
if $\hat{\cS}_i = \sL$,
then $\hat{\cS}_{i+\hat{d}} = \sL$
for all but one value of $i \in \{ 0,\ldots,n-1 \}$
(where $\hat{d} \hat{m} = 1$).
We now show this property of $\hat{\cS}$ is not exhibited by $\check{\cS}$
and hence $\check{\cS}$ and $\hat{\cS}$ cannot be equal.

Using (\ref{eq:cScheck2}) and remembering $\hat{d},l \ne 1,n-1$,
if $\hat{d} \le l$, then
$\check{\cS}_{l-\hat{d}} = \check{\cS}_{l-\hat{d}+1} = \sL$ and
$\check{\cS}_l = \check{\cS}_{l+1} = \sR$.
Similarly if $\hat{d} > l$, we have
$\check{\cS}_0 = \check{\cS}_1 = \sL$ and
$\check{\cS}_{\hat{d}} = \check{\cS}_{\hat{d}+1} = \sR$.
In either case
\begin{equation}
\check{\cS} \ne \hat{\cS}
\end{equation}
for any $k \in \mathbb{Z}$.
\hfill

\end{enumerate}
\end{proof}

\section{The cardinality of symbol sequences}
\label{sec:CARD}

Let $\mathcal{N}_n$ [$\mathcal{N}_n^{\rm rot}$]
denote the number of primitive \sew s [primitive rotational \sew s] 
of length $n$ that are distinct up to cyclic permutation.
To use combinatorics terminology, $\mathcal{N}_n$
is the number of $n$-bead necklaces of two colors
with primitive period $n$, and it is the number of binary
{\em Lyndon words} of length $n$ \cite{Lo83,HaZh98}.
The formulas we give for $\mathcal{N}_n$ and $\mathcal{N}_n^{\rm rot}$
use the \mob~function:
\begin{equation}
\mu(n) = \left\{ \begin{array}{ll}
1 & {\rm if~} n=1 {\rm ~or~} n {\rm ~is~the~product~of~an~even~number~of~primes} \\
-1 & {\rm if~} n {\rm ~is~the~product~of~an~odd~number~of~primes} \\
0 & {\rm otherwise,~i.e.~} n {\rm ~is~not~square~free}
\end{array} \right.
\end{equation}
and Euler's totient function:
\begin{equation}
\varphi(n) = {\rm ~the~number~of~positive~integers~less~than~} n
{\rm ~and~coprime~to~} n
\end{equation}
where $\varphi(1) = 1$. The first few values of $\mu(n)$ and $\varphi(n)$ are:
\begin{center}
\begin{tabular}{c|cccccccccccc}
n & 1 & 2 & 3 & 4 & 5 & 6 & 7 & 8 & 9 & 10 & 11 & 12 \\ \hline
$\mu(n)$ & 1 & -1 & -1 & 0 & -1 & 1 & -1 & 0 & 0 & 1 & -1 & 0 \\ \hline
$\varphi(n)$ & 1 & 1 & 2 & 2 & 4 & 2 & 6 & 4 & 6 & 4 & 10 & 4
\end{tabular}
\end{center}
Then
\begin{eqnarray}
\mathcal{N}_n & = & \frac{1}{n} \sum_{a|n} \mu(\frac{n}{a}) 2^a
\label{eq:numSS} \\
\mathcal{N}_n^{\rm rot} & = & 2 + \frac{n-3}{2} \varphi(n) \;, {\rm ~~for~} n \ge 3
\label{eq:numOSS}
\end{eqnarray}
where $\sum_{a|n}$ denotes summation over all divisors, $a$, of $n$.
The first formula, (\ref{eq:numSS}), is well-known
(see for instance \cite{Lo83,HaZh98} for a derivation).
To arrive at (\ref{eq:numOSS}) we
count the number of distinct (up to cyclic permutation) $\cS[l,m,n]$
for an arbitrary fixed $n$:
When $l = 1$ there is one distinct rotational \sew,
namely $\sL \sR^{n-1}$.
Similarly when $l = n-1$ there is only $\sL^{n-1} \sR$.
There are $(n-3)$ values of $l$ left to consider.
For each of these, there are $\varphi(n)$ possible values for $m$.
By Lem.~\ref{le:oss}\ref{it:notCycPerS},
the $\frac{\varphi(n)}{2}$ values of $m$ that are less than
$\frac{n}{2}$ yield distinct \sew s and by Lem.~\ref{le:oss}\ref{it:cycPerS},
the remaining values of $m$ only produce cyclic permutations.
Thus we have (\ref{eq:numOSS}).

The first few values of $\mathcal{N}_n$ and $\mathcal{N}_n^{\rm rot}$ are:
\begin{center}
\begin{tabular}{c|cccccccccccc}
n & 1 & 2 & 3 & 4 & 5 & 6 & 7 & 8 & 9 & 10 & 11 & 12 \\ \hline
$\mathcal{N}_n$ & 2 & 1 & 2 & 3 & 6 & 9 & 18 & 30 & 56 & 99 & 186 & 335 \\ \hline
$\mathcal{N}_n^{\rm rot}$ & 0 & 1 & 2 & 3 & 6 & 5 & 14 & 12 & 20 & 16 & 42 & 20
\end{tabular}
\end{center}
All primitive \sew s of length $n < 6$ are rotational \sew s
except when $n = 1$ because we do not consider $\sL$ and $\sR$ to be rotational.
There are four distinct primitive \sew s
of length six that are not rotational. 
These are: ${\sf LRLRRR}$, ${\sf LLRLRR}$, ${\sf LLRRLR}$ and ${\sf LLLRLR}$.
Roughly, as $n$ increases the number of distinct primitive \sew s of length $n$
that are not rotational also increases.
To quantify this statement, $\mathcal{N}_n$ (\ref{eq:numSS}),
grows like ${\rm e}^n$, whereas $\mathcal{N}_n^{\rm rot}$ (\ref{eq:numOSS}),
grows like $n^2$ (since $\varphi(n)$ grows linearly
($\varphi(p) = p-1$ for any prime $p$)).
Thus for large $n$ the majority of primitive \sew s are not rotational.

\section{Shrinking points}
\label{sec:SHRINK}

Roughly speaking, as in \cite{YaHa87},
we call points where lens-chain shaped resonance tongues
have zero width, shrinking points.
The aim of this section is to provide a rigorous foundation
for the unfolding of shrinking points described in \S\ref{sec:UNFOLD}.
We define two classes of shrinking points: terminating and non-terminating.
This categorization provides a distinction between shrinking points that
lie at the end of a lens-chain (terminating)
and those that lie in the middle (non-terminating)
(recall Fig.~\ref{fig:res}).
We assume associated \sew s are rotational (see \S\ref{sec:ROTAT});
this assumption is crucial to our analysis.
Our main results are
Lem.~\ref{le:shrinkPoly}, which states
that at a shrinking point there exists an invariant, nonplanar
(though planar in special cases) polygon,
and Cor.~\ref{co:shrink1}, which
shows that shrinking points are a hub for the singularity of
important matrices,

Near a shrinking point, a lens shaped resonance tongue
corresponding to the existence of an admissible, stable
period-$n$ cycle, $\{ x_i \}$, has two boundaries.
These correspond to border-collision fold bifurcations of $\{ x_i \}$
with an unstable orbit of the same period.
At each boundary, one point on the orbit lies on the \sw.
The first question to address is the following:
which points lie on the \sw~at the two resonance tongue boundaries?

Suppose $\{ x_i \}$ has an associated \sew~that
is rotational, $\cS[l,m,n]$.
We may picture the points, $x_i$, lying on a topological circle,
as in Fig.~\ref{fig:rotSymSeq}.
The switching manifold intersects the circle at two points.
If this structure is maintained as parameters vary,
then it seems reasonable that only points that
lie adjacent to an intersection 
can collide with the switching manifold.
When $l \ne 1, n-1$, there are four such points:
$x_0$, $x_{-d}$, $x_{(l-1)d}$ and $x_{ld}$.
(For example, if $[l,m,n] = [3,2,7]$, as in Fig.~\ref{fig:rotSymSeq},
then $d = 4$ and the four adjacent points are
$x_0$, $x_3$, $x_1$ and $x_5$.)

Suppose w.l.o.g.~that $x_0$ lies on the \sw~at one resonance tongue boundary.
In the interior of the tongue the
\sew~of the corresponding unstable periodic solution
then differs from $\cS$ in
the $0^{\rm th}$ element, i.e., is $\cS^{\overline{0}}$.
If we assume that $\{ x_i \}$ collides
and annihilates with the same unstable periodic solution
on the second boundary, it must be $x_{(l-1)d}$ that lies on the
\sw~there because the only index $i$ for which $\cS^{\overline{id}}$
is a cyclic permutation of $\cS^{\overline{0}}$, is $i = l-1$.
Consequently, in view of Lem.~\ref{le:bc} since
$\{ x_i \}$ is not always well-defined,
we define non-\tsp s by the singularity of
$P_\cS$ and $P_{\cS^{((l-1)d)}}$
(Def.~\ref{def:shrink} below).
In \S\ref{sec:UNFOLD} we will show
that resonance tongue boundaries at which the remaining two points,
$x_{-d}$ and $x_{ld}$, lie on the \sw,
form a second lens shaped resonance tongue
emanating from the shrinking point.

As we will see, the cases $l = 1$ and $l = n-1$
correspond to \tsp s.
It suffices to consider only one of these cases,
we choose $l = n-1$,
because they are interchangeable via swapping
$\sL$ and $\sR$.
For the remainder of this paper we assume $l \ne 1$
which greatly simplifies our analysis.
In particular we find it is always reasonable to assume
that the unstable periodic solution described above
always exists (though it may not be admissible).
That is $(I-M_{\cS^{\overline{0}}})$ is nonsingular
and therefore the \ncss, (\ref{eq:nCycleSolutionSystem}),
of $\cS^{\overline{0}}$ has the unique solution
\begin{equation}
p = p_0 = \mu (I-M_{\cS^{\overline{0}}})^{-1} P_{\cS^{\overline{0}}} b \;.
\label{eq:pDef}
\end{equation}
We denote the $i^{\rm th}$ iterate of $p$ via the
\sew~$\cS^{\overline{0}}$ by $p_i$
and let $t_i$ denote its first component.
The orbit, $\{ p_i \}$, will play a pivotal role in our analysis.

\begin{definition}[Non-Terminating Shrinking Point]~\\
Consider the map (\ref{eq:pwaMap}) with $N \ge 2$
and suppose that $\mu \ne 0$ and $\varrho^{\sf T} b \ne 0$.\\
Let $\cS = \cS[l,m,n]$ be a rotational \sew~with $1 < l < n-1$.\\
Suppose
\begin{equation}
P_\cS {\rm ~and~} P_{\cS^{((l-1)d)}} {\rm ~are~singular.}
{\rm ~~~~~~(the~singularity~condition)}
\nonumber
\end{equation}
Let $\check{\cS} = \cS^{\overline{0}}$
and $\hat{\cS} = \cS^{\overline{ld}}$
and assume $(I-M_{\check{\cS}})$ and $(I-M_{\hat{\cS}})$ are nonsingular.\\
Suppose the orbit, $\{ p_i \}$, of (\ref{eq:pDef}), is admissible.\\
Then we say (\ref{eq:pwaMap}) is at a {\em non-\tsp}.
\label{def:shrink}
\end{definition}

Def.~\ref{def:shrink} essentially characterizes non-\tsp s
as the codimension-two phenomenon at which
the border-collision matrices $P_\cS$ and $P_{\cS^{((l-1)d)}}$ are
simultaneously singular.
The following definition for \tsp s is quite different
because the defining characteristic of these points,
in our opinion,
is the codimension-two requirement that one fixed point
(here $x^{*(\sL)}$ because we are assuming $l = n-1$) is admissible and
has a pair of associated multipliers on the unit circle with a particular
rational angular frequency.

\begin{definition}[Terminating Shrinking Point]~\\
Consider the map (\ref{eq:pwaMap}) with $N \ge 2$,
suppose $(I-A_\sL)$ is nonsingular
and $\frac{\mu \varrho^{\sf T} b}{\det(I-A_\sL)} < 0$
(i.e.~the fixed point $x^{*(\sL)}$ is admissible, see (\ref{eq:xStari1})).\\
Let $\cS = \cS[l,m,n]$ be a rotational \sew~with $l = n-1$ and $n \ge 3$.\\
Let $\check{\cS} = \cS^{\overline{0}}$
and suppose $(I-M_{\check{\cS}})$ is nonsingular.\\
Suppose
\begin{equation}
{\rm e}^{\pm \frac{2 \pi {\rm i} m}{n}} {\rm ~are~multipliers~of~} A_\sL \;.
{\rm ~~~~~~(the~singularity~condition)}
\nonumber
\end{equation}
Then we say (\ref{eq:pwaMap}) is at a
{\em \tsp}.
\label{def:tershr}
\end{definition}

\begin{figure}[ht]
\begin{center}
\includegraphics[width=5cm,height=6cm]{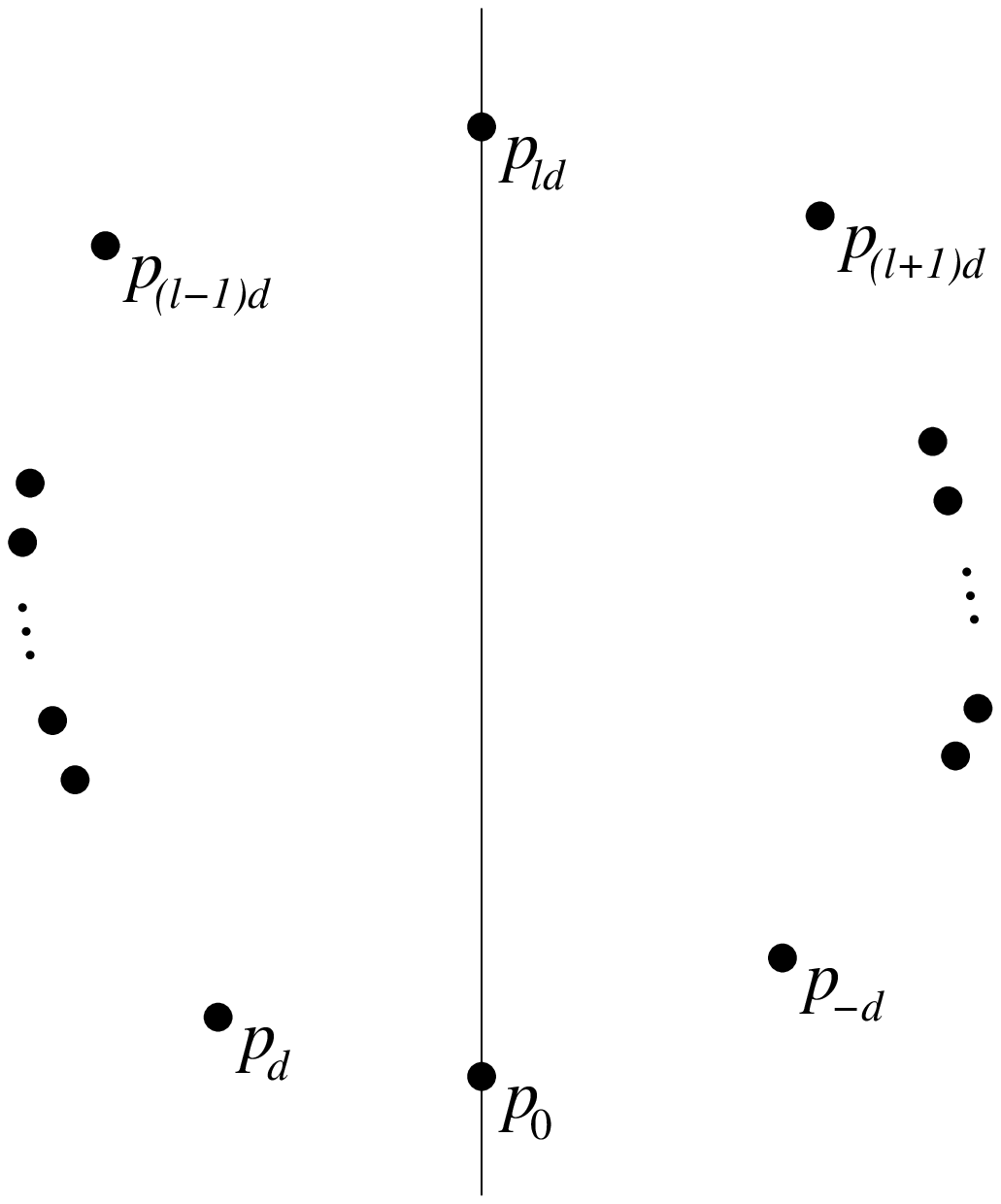}
\end{center}
\caption{
A schematic illustrating iterates of $p$, (\ref{eq:pDef}),
on or near the \sw~at a non-\tsp.
}
\label{fig:pSchem}
\end{figure}

In two dimensions, \tsp s are center bifurcations
for rational rotation numbers, $m/n$,
studied in \cite{SuGa08,SuGa06}.
These authors show there exists an invariant polygon
that has one side on the \sw, within which all points
other than the fixed point belong to periodic orbits with
rotation number, $m/n$.
As we will show, in higher dimensions this behavior occurs on the
center manifold of $x^{*(\sL)}$ corresponding to the multipliers
${\rm e}^{\pm \frac{2 \pi {\rm i} m}{n}}$, call it $E_c$.

The following lemma concerns the orbit, $\{ p_i \}$, of (\ref{eq:pDef}).
The reader should take care to notice that
singularity of the matrix $P_{\cS^{((l-1)d)}}$,
by Lem.~\ref{le:bc},
implies that $p_{ld}$ (not $p_{(l-1)d}$) lies on the \sw~because
\begin{equation}
P_{\check{\cS}^{(ld)}} = P_{\cS^{\overline{0}(ld)}} =
P_{\cS^{((l-1)d)\overline{0}}} = P_{\cS^{((l-1)d)}} \;,
\label{eq:equalPs}
\end{equation}
where the second equality is Lem.~\ref{le:oss}\ref{it:switchS}
and (\ref{eq:PindepS0}) is used for the last equality.

\begin{figure}[h]
\begin{center}
\includegraphics[width=7.2cm,height=6cm]{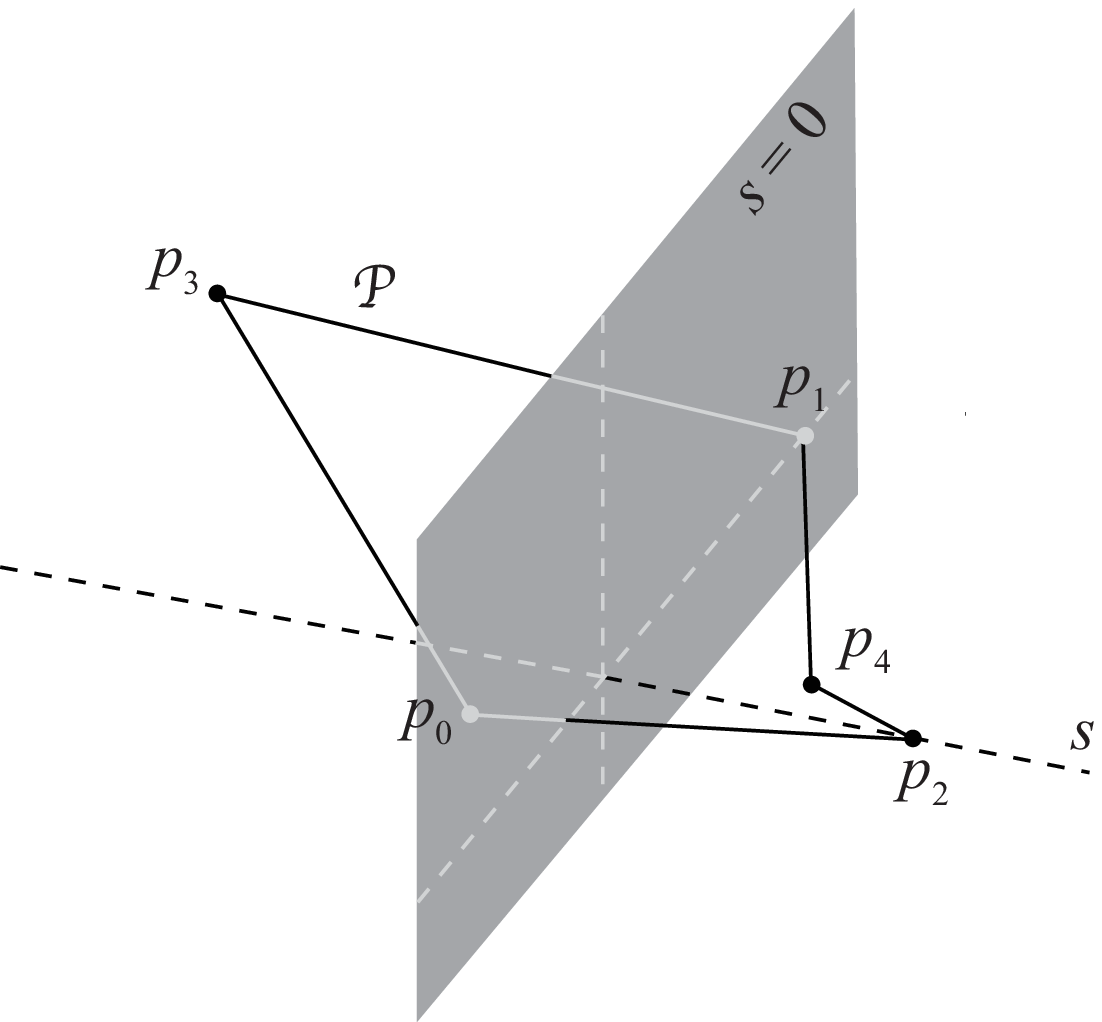}
\end{center}
{Figure 6: The invariant, nonplanar polygon, $\mathcal{P}$,
at a non-\tsp~corresponding to
the rotational \sew~$\mathcal{S}[2,2,5] = {\sf LRRLR}$
for the map (\ref{eq:pwaMap}) when
$A_\sL = \begin{bmatrix}
			0 & 1 & 0 \\
			1 & 0 & 1 \\
			\frac{28}{87} & 0 & 0
		 \end{bmatrix}$,
$A_\sR = \begin{bmatrix}
			-\frac{23}{14} & 1 & 0 \\
			0 & 0 & 1 \\
			\frac{3}{2} & 0 & 0
		\end{bmatrix}$,
$b = e_1$ and $\mu = 1$.
Here $p_0 = (0,-1,\frac{3}{2})^{\sf T}$.
The \sw, $s = 0$, is shaded gray.}
\label{fig:pent3d}
\end{figure}
\addtocounter{figure}{1}

\begin{lemma}
Suppose (\ref{eq:pwaMap}) is at a shrinking point.
Then,
\begin{enumerate}[label=\alph{*}),ref=\alpha{*}]
\item $t = t_{ld} = 0$;
\item if the shrinking point is non-terminating,
then $t_d, t_{(l-1)d} < 0$ and $t_{(l+1)d}, t_{-d} > 0$
as in Fig.~\ref{fig:pSchem};
if the shrinking point is terminating,
then $t_{id} < 0$ for all $i \ne 0,-1$;
\item $\{ p_i \}$ has period $n$;
\item the iterates of $p$ are the vertices of an invariant, nonplanar
$n$-gon, $\mathcal{P}$, that is comprised of uncountably many $\cS$-cycles
and the restriction of (\ref{eq:pwaMap}) to $\mathcal{P}$
is homeomorphic to rigid rotation with rotation number, $m/n$.
\end{enumerate}
\label{le:shrinkPoly}
\end{lemma}

Separate proofs for non-terminating and \tsp s
are given in Appendix \ref{sec:PROOFSb}.
If the shrinking point is terminating, then the polygon $\mathcal{P}$ is planar.
In general, when $N > 2$, $\mathcal{P}$ is nonplanar as in Fig.~6.   
In any case $\mathcal{P}$ has two vertices on the \sw.

$\mathcal{P}$ is comprised of many $\cS$-cycles,
in other words the \ncss~of $\cS$ has more than one distinct solution.
Therefore at a shrinking point the matrix $(I-M_\cS)$ is singular.
Furthermore, for each $i$, by Lem.~\ref{le:cyclicDetM},
$(I-M_{\cS^{(i)}})$ is singular.
However, there exist solutions to the \ncss~of $\cS^{(i)}$
(such as $p_i$), so to avoid a contradiction with Lem.~\ref{le:sn},
$P_{\cS^{(i)}}$ must be singular.
Consequently we have the following:

\begin{corollary}
Suppose (\ref{eq:pwaMap}) is at a shrinking point.
Then,
\begin{enumerate}[label=\alph{*}),ref=\alpha{*}]
\item $(I-M_\cS)$ is singular,
\item $P_{\cS^{(i)}}$ is singular, for all $i$.
\end{enumerate}
\label{co:shrink1}
\end{corollary}

The singularity of a matrix is a codimension-one phenomenon.
Thus by Cor.~\ref{co:shrink1}, we expect there will be
curves in two-dimensional parameter space passing through 
shrinking points along which $(I-M_\cS)$ and each $P_{\cS^{(i)}}$
is singular.
Some of these curves form resonance tongue boundaries,
and it is these that form the focus of the following section.

\section{Unfolding shrinking points}
\label{sec:UNFOLD}

This section presents an unfolding of the dynamics near terminating and non-\tsp s.
We begin by assuming (\ref{eq:pwaMap}) is parameter dependent
and perform a coordinate transformation such that, locally,
in two dimensions, two tongue boundaries lie on the positive
coordinate axes.

Suppose (\ref{eq:pwaMap}) varies with
a set of parameters $\xi \in \mathbb{R}^K$ ($K \ge 2$) and is $C^{k+1}$ ($k \ge 2$).
Then $A_\sL$ and $A_\sR$ are $C^k$ functions of $\xi$.
Suppose (\ref{eq:pwaMap}) is at a shrinking point 
when $\xi = 0$, for some fixed $\mu \ne 0$.
Then, $\det(P_\cS(\xi))$ and $\det(P_{\cS^{((l-1)d)}}(\xi))$ are $C^k$
scalar functions that have the value zero when $\xi = 0$.
If the $2 \times 2$ Jacobian
$\frac{\partial(\det(P_\cS),\det(P_{\cS^{((l-1)d)}}))}{\partial( \xi_i, \partial \xi_j)}$
is nonsingular for some $i^{\rm th}$ and $j^{\rm th}$ components
of $\xi$, we may utilize the implicit function theorem
to obtain a coordinate transformation of $\xi$, such that
$\det(P_\cS) = 0$ when $\eta = \xi_i = 0$
and $\det(P_{\cS^{((l-1)d)}}) = 0$ when $\nu = \xi_j = 0$.
In what follows below we omit dependence of (\ref{eq:pwaMap})
on the remaining $(K-2)$ components of $\xi$.

Let $\check{x}(\eta,\nu)$ denote the unique solution to the
\ncss, (\ref{eq:nCycleSolutionSystem}),
of $\check{\cS} = \cS^{\overline{0}}$ for small values of $\eta$ and $\mu$.
Let $p = \check{x}(0,0)$ (to coincide with (\ref{eq:pDef})).
If the shrinking point is non-terminating,
let $\hat{x}(\eta,\nu)$ denote the unique solution to the
\ncss~of $\hat{\cS} = \cS^{\overline{ld}}$ for small values of $\eta$ and $\nu$.
Let $\check{s}_i$ and $\hat{s}_i$ denote the
first components of $\check{x}_i$ and $\hat{x}_i$, respectively.
From the above assumptions on $\eta$ and $\mu$,
by Lem.~\ref{le:bc}
we have $\check{s}(0,\nu) = 0$ and $\check{s}_{ld}(\eta,0) = 0$,
(see (\ref{eq:equalPs})).
If $\frac{\partial}{\partial \eta} \check{s}(0,0)$
and $\frac{\partial}{\partial \nu} \check{s}_{ld}(0,0)$
are both nonzero,
we may redefine $\eta$ and $\nu$ via a nonlinear scaling so that
\begin{eqnarray}
\check{s}(\eta,\nu) & = & \eta(1 + O(1)) \label{eq:nuAxis} \;, \\
\check{s}_{ld}(\eta,\nu) & = & \nu(1 + O(1)) \label{eq:etaAxis} \;. 
\end{eqnarray}
Note, we use $O(k)$ (and $o(k)$)
to denote terms that are order $k$ or larger (larger than order $k$)
in all variables and parameters.

\begin{theorem}
Suppose the map (\ref{eq:pwaMap}) is a $C^{k+1}$ ($k \ge 2$)
function of parameters $\eta$ and $\nu$,
has a shrinking point when $\eta = \nu = 0$
and we have (\ref{eq:nuAxis}) and (\ref{eq:etaAxis}).
Write
\begin{equation}
\det(I-M_\cS(\eta,\nu)) = k_1 \eta + k_2 \nu + O(2) \;,
\label{eq:detMTS}
\end{equation}
and assume $k_1,k_2 \ne 0$. As usual we assume $N \ge 2$.

Let $\Psi_1 = \{ (\eta,\nu) ~|~ \eta, \nu \ge 0 \}$.
Then for small $\eta$ and $\nu$,
the $\cS$ and $\check{\cS}$-cycles are admissible in $\Psi_1$
and collide in border-collision fold bifurcations
at the boundaries, $\eta = 0$ and $\nu = 0$.

If the shrinking point is non-terminating,
there exist $C^k$ functions 
$g_1, g_2 : \mathbb{R} \to \mathbb{R}$
such that for small $\eta$ and $\nu$,
the $\cS$ and $\hat{\cS}$-cycles are admissible in
$\Psi_2 = \{ (\eta,\nu) ~|~ \eta \le g_1(\nu), \nu \le g_2(\eta) \}$ and
collide in border-collision fold bifurcations
at the boundaries, $\eta = g_1(\nu)$ and $\nu = g_2(\eta)$.
Furthermore, the lowest-order, nonzero terms of $g_1$ and $g_2$ are
$g_1''(0), g_2''(0) < 0$.
\label{th:shrink}
\end{theorem}

Near non-\tsp s, Th.~\ref{th:shrink}
predicts the bifurcation set sketched in Fig.~\ref{fig:shrinkSchematic}.
The four resonance tongue boundaries that emanate from the
shrinking point correspond to the collision of
the points of the $\cS$-cycle:
$x_0$, $x_{-d}$, $x_{(l-1)d}$, $x_{ld}$, with the \sw.
As mentioned in \S\ref{sec:SHRINK} these are the four points
adjacent to the \sw~in the interior of the tongue.

By applying the implicit function theorem to (\ref{eq:detMTS})
there exists a $C^k$ function
\begin{equation}
h(\eta) = -\frac{k_1}{k_2} \eta + O(\eta^2) \;,
\label{eq:shrinkH}
\end{equation}
such that $\det(I-M_\cS(\eta,h(\eta))) = 0$.
But by Th.~\ref{th:shrink},
$\cS$-cycles exist throughout the first quadrant of parameter space
therefore the curve $\nu = h(\eta)$ must not enter this quadrant.
Hence, since $h'(0) = -\frac{k_1}{k_2}$,
\begin{equation}
k_1 k_2 < 0 \;,
\label{eq:k1k2}
\end{equation}
as in Fig.~\ref{fig:shrinkSchematic}.

By Cor.~\ref{co:shrink1}, the border-collision matrix
of every cyclic permutation of $\cS$ is singular at a shrinking point.
Th.~\ref{th:shrink} describes curves along which
some of these matrices remain singular.
We now describe all other such curves.
If $p_0$ and $p_{ld}$ are the only points of the period-$n$ cycle, $\{ p_i \}$,
that lie on the \sw~(as is generically the case),
then for each $i \ne 0, l-1, l, -1$ if the shrinking point is non-terminating
and for each $i \ne 0, -2, -1$ if the shrinking point is terminating,
there exists a $C^k$ function
\begin{equation}
q_i(\eta) = -\frac{k_1 t_{(i+1)d}}{k_2 t_{id}} \eta + O(\eta^2) \;,
\end{equation}
that satisfies $\det(P_{\cS^{(id)}}(\eta,q_i(\eta))) = 0$.
In each case $t_{(i+1)d}$ and $t_{id}$ have the same sign,
thus by (\ref{eq:k1k2}) we have $q_i'(0) < 0$.
Although the border-collision matrix, $P_{\cS^{(id)}}$, is singular
along $\nu = q_i(\eta)$,
these curves are not seen in resonance tongue diagrams because
the related periodic solutions are virtual.

So far we have described curves passing
through the shrinking point along which $P_{\cS^{(id)}}$ is singular
for each $i$, except for $i = -1$ when the shrinking point is terminating.
Here $\cS^{(-d)} = \sR \sL^{n-1}$, hence
$P_{\cS^{(-d)}} = I + A_\sL + \cdots + A_\sL^{n-1}
= (I-A_\sL)^{-1} (I-A_\sL^n)$.
The matrix $A_\sL^n$ has a multiplier 1 with an algebraic multiplicity
of at least two, therefore $\det(P_{\cS^{(-d)}}(\eta,\nu)) = O(2)$.
Thus in this case, for small $\eta$ and $\nu$,
generically $P_{\cS^{(-d)}}$ is singular only when $\eta = \nu = 0$. 

In addition, for \tsp s there generically exists a curve in parameter
space passing through the shrinking point along which
the matrix, $A_\sL$, has a pair of complex multipliers on the unit circle.
As one moves along the curve, the angular frequency of multipliers changes.
Whenever the angular frequency is rational there may be additional \tsp s.
Such curves are described in \cite{SiMe08b,SuGa08,SuGa06,ZhMo08b}.

The border-collision bifurcations occurring on
resonance tongue boundaries are described by
piecewise-affine maps of the form (\ref{eq:nthPWA}).
Th.~\ref{th:shrink} states that these bifurcations are nonsmooth folds,
thus we may use Feigen's results \cite{DiFe99}, summarized in \S\ref{sec:FRAME},
to determine the relative stability of periodic solutions
near shrinking points.
Let $a_i[\cS]$ denote the number of real multipliers of $M_\cS$
that are greater than one in the interior of $\Psi_i$.
Then $a_1[\cS] + a_1[\check{\cS}]$ and
$a_2[\cS] + a_2[\hat{\cS}]$ are odd.
Also, $\det(I-M_\cS(\eta,\nu)) = O(1)$, (\ref{eq:detMTS}),
thus $|a_1[\cS] - a_2[\cS]| = 1$.
Consequently, if $\cS$-cycles are stable in the interior of $\Psi_1$,
then $\check{\cS}$-cycles are unstable and
$\cS$-cycles are unstable in the interior of $\Psi_2$.

\begin{figure}[ht]
\begin{center}
\includegraphics[width=8.4cm,height=7cm]{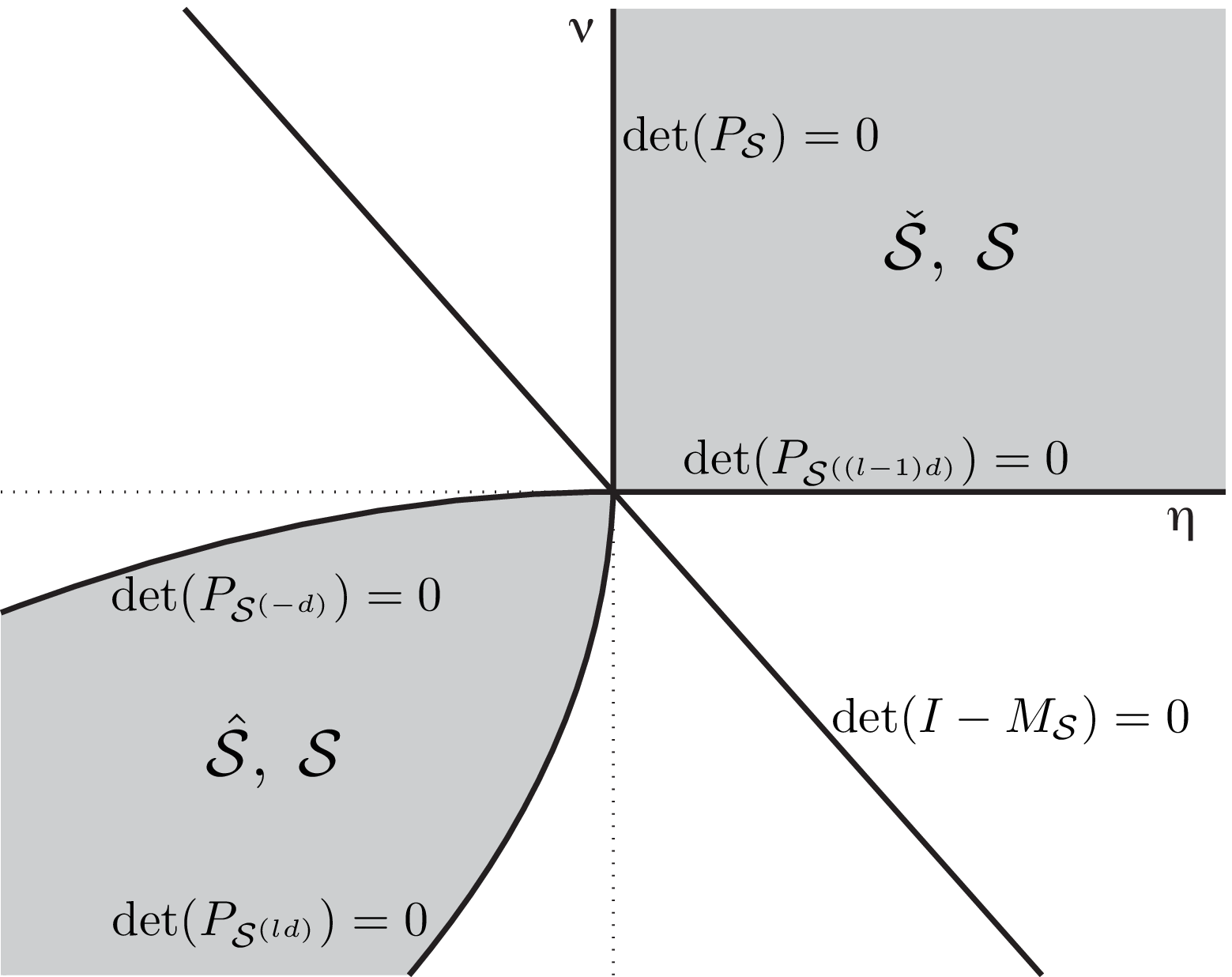}
\end{center}
\caption{
A schematic of the generic bifurcation set in a neighborhood of a non-\tsp.
}
\label{fig:shrinkSchematic}
\end{figure}

Finally we turn to the proof of Th.~\ref{th:shrink}.

\begin{proof}[Proof of Th.~\ref{th:shrink}]
Write
\begin{eqnarray}
\det(I-M_{\check{\cS}}(\eta,\nu)) & = & k_3 + O(1) \label{eq:detCMTS} \;, \\
\det(I-M_{\hat{\cS}}(\eta,\nu)) & = & k_4 + O(1) \label{eq:detHMTS} \;,
\end{eqnarray}
where $k_3, k_4 \ne 0$ by assumption.
Let
\begin{equation}
x(\eta,\nu) = \mu (I-M_\cS(\eta,\nu))^{-1} P_\cS(\eta,\nu) b \;,
\end{equation}
whenever $\nu \ne h(\eta)$, (\ref{eq:shrinkH}).
For small $\nu \ne 0$, if $\eta = 0$ then $x$ is defined
and coincides with $\check{x}$ because $\det(P_\cS) = 0$.
Therefore $x_{id}(0,\nu) = \check{x}_{id}(0,\nu)$, for all $i$.
In particular,
\begin{equation}
s_{id}(0,\nu) = t_{id} + O(\nu) \;,
\label{eq:x1id1}
\end{equation}
so that by (\ref{eq:nuAxis}) and (\ref{eq:etaAxis})
\begin{eqnarray}
s(0,\nu) & = & 0 \nonumber \;, \\
s_{ld}(0,\nu) & = & \nu + O(\nu^2) \label{eq:x1ld1} \;.
\end{eqnarray}
Similarly for small $\eta \ne 0$, if $\nu = 0$ then $x$ is defined
and $\det(P_{\cS^{((l-1)d)}}) = 0$, thus $x$ solves the \ncss~of
$\cS^{\overline{(l-1)d}} = \check{\cS}^{(d)}$.
Therefore $x_{id}(\eta,0) = \check{x}_{(i+1)d}(\eta,0)$, for all $i$.
In particular,
\begin{equation}
s_{id}(\eta,0) = t_{(i+1)d} + O(\eta) \;,
\label{eq:x1id2}
\end{equation}
so that
\begin{eqnarray}
s_{-d}(\eta,0) & = & \eta + O(\eta^2) \nonumber \;, \\
s_{(l-1)d}(\eta,0) & = & 0 \nonumber \;.
\end{eqnarray}
We now compute lowest order terms in the Taylor series
of $\hat{s}_{ld}(\eta,\nu)$.
An application of Lem.~\ref{le:detPair} to the \sew~$\cS^{(ld)}$,
using also Lem.~\ref{le:cyclicDetM}, produces
\begin{equation}
\det(I-M_\cS(\eta,\nu)) s_{ld}(\eta,\nu) =
\det(I-M_{\hat{\cS}}(\eta,\nu)) \hat{s}_{ld}(\eta,\nu) \nonumber \;.
\end{equation}
Using (\ref{eq:detMTS}), (\ref{eq:detHMTS}), (\ref{eq:x1ld1}) and (\ref{eq:x1id2})
we obtain
\begin{eqnarray}
\hat{s}_{ld}(0,\nu) & = & \frac{k_2}{k_4} \nu^2 + O(\nu^3) \nonumber \;, \\
\hat{s}_{ld}(\eta,0) & = &
\frac{k_1 t_{(l+1)d}}{k_4} \eta + O(\eta^2) \label{eq:x1ld3} \;.
\end{eqnarray}
By the implicit function theorem,
there exists a unique $C^k$ function $g_1 : \mathbb{R} \to \mathbb{R}$
such that for small $\nu$,
$\hat{s}_{ld}(g_1(\nu),\nu) = 0$ and
\begin{equation}
g_1(\nu) = -\frac{k_2}{k_1 t_{(l+1)d}} \nu^2 + O(\nu^3) \;.
\label{eq:g1}
\end{equation}
In a similar fashion, by applying Lem.~\ref{le:detPair} to the
\sew~$\cS^{(-d)}$ and noting $\cS^{(-d)\overline{0}} = \hat{\cS}$,
we obtain
\begin{eqnarray}
\hat{s}(0,\nu) & = & \frac{k_2 t_{-d}}{k_4} \nu + O(\nu^2) \label{eq:x104} \;, \\
\hat{s}(\eta,0) & = & \frac{k_1}{k_4} \eta^2 + O(\eta^3) \nonumber \;,
\end{eqnarray}
so that the implicit function theorem gives the function
\begin{equation}
g_2(\eta) = -\frac{k_1}{k_2 t_{-d}} \eta^2 + O(\eta^3) \;.
\label{eq:g2}
\end{equation}

We now prove $k_1 k_2 > 0$.
Let $\Sigma_i \subset \mathbb{R}^2$ denote the intersection of the
interior of the $i^{\rm th}$ quadrant with a sufficiently small
neighborhood of the origin ($i = 1,\ldots,4$).
Suppose for a contradiction, $k_1 k_2 < 0$.
Then, by (\ref{eq:detMTS}), $(I-M_\cS)$ is nonsingular throughout $\Sigma_2$
and hence $s$ is continuous throughout $\Sigma_2$.
Notice $s = 0$ only when $\eta = 0$, thus the sign of $s$
is constant in $\Sigma_2$.
When $\nu = 0$, by admissibility and using (\ref{eq:x1id2})
we have $s(\eta,0) = t_d + O(1) < 0$,
thus $s$ is negative throughout $\Sigma_2$.
In particular, $s$ is negative when 
$\nu = \frac{k_1 \eta}{k_2}$ for small $\eta < 0$.
By applying Lem.~\ref{le:detPair} to $\cS$, we obtain
\begin{equation}
\det(I-M_\cS(\eta,\nu)) s(\eta,\nu) =
\det(I-M_{\check{\cS}}(\eta,\nu)) \check{s}(\eta,\nu) \;.
\label{eq:dr3}
\end{equation}
Substituting $\nu = \frac{k_1 \eta}{k_2}$ and using
(\ref{eq:nuAxis}), (\ref{eq:detMTS}) and (\ref{eq:detCMTS}) gives
\begin{eqnarray}
(2 k_1 \eta + O(\eta^2)) s(\eta,\frac{k_1 \eta}{k_2}) & = &
(k_3 + O(\eta))(\eta + O(\eta^2)) \nonumber \;, \\
\Rightarrow~~~2 k_1 s(\eta,\frac{k_1 \eta}{k_2}) & = & k_3 + O(\eta) \nonumber \;, \\
\Rightarrow~~~k_1 k_3 & < & 0 \label{eq:contradict1} \;.
\end{eqnarray}
Via a similar argument with $s_{(l-1)d}$,
we find $k_2 k_3 < 0$.
This inequality provides a contradiction with (\ref{eq:contradict1}), thus $k_1 k_2 > 0$.
Hence $s$ and $s_{(l-1)d}$ are in fact continuous
and negative throughout $\Sigma_1$ and $\Sigma_3$, thus $k_1 k_3 < 0$.
Via similar arguments we find $k_1 k_4 > 0$.
Consequently we have
\begin{equation}
{\rm sgn}(k_1) = {\rm sgn}(k_2) = -{\rm sgn}(k_3) = {\rm sgn}(k_4)
\label{eq:allK}
\end{equation}
Furthermore, from (\ref{eq:g1}) and (\ref{eq:g2}) we have
\begin{equation}
g_1''(0), g_2''(0) < 0 \;,
\end{equation}
since $t_{(l+1)d}, t_{-d} > 0$.

Near $(\eta,\nu) = (0,0)$,
since $\{ p_i \}$ is admissible,
$\check{\cS}$-cycles are admissible if and only if
$\check{s}, \check{s}_{ld} \ge 0$,
thus only in $\Psi_1$ by (\ref{eq:nuAxis}) and (\ref{eq:etaAxis}).
Similarly $\hat{\cS}$-cycles are admissible only if
$\hat{s}, \hat{s}_{ld} \le 0$.
By (\ref{eq:x104}),
$\frac{\partial \hat{s}}{\partial \nu}(0,0) = \frac{k_2 t_{-d}}{k_4}$
is positive by (\ref{eq:allK}) and admissibility.
Thus $\hat{s}(\eta,\nu) \le 0$ when $\nu \le h_2(\eta)$.
Similarly, by (\ref{eq:x1ld3}),
$\hat{s}_{ld}(\eta,\nu) \le 0$ when $\eta \le h_1(\nu)$.
Therefore $\hat{\cS}$-cycles  are admissible in $\Psi_2$.
The result for $\cS$-cycles follows by
looking at the signs of each $s_i$ on
the boundaries of $\Psi_1$ and $\Psi_2$.
\hfill
\end{proof}

\section{Discussion}
\label{sec:CONC}

Dynamics local to a nondegenerate border-collision bifurcation of a fixed point
of a piecewise-smooth, continuous map are described by a piecewise-affine
approximation of the form (\ref{eq:pwaMap}).
The border-collision bifurcation occurs when the parameter, $\mu$, is zero.
As $\mu$ is varied from zero, invariant sets such as periodic solutions
and topological circles
may emanate from the bifurcation.
The focus of this paper has been to fix $\mu$ at some nonzero value,
assume (\ref{eq:pwaMap}) varies with other independent parameters,
and investigate the nature of regions in parameter space
where periodic solutions exist.
In two-dimensional parameter space these resonance tongues
commonly display a lens-chain structure.
This paper has performed a rigorous unfolding about
points for which resonance tongues have zero width - shrinking points.

Curves corresponding to four different border-collision fold bifurcations
form resonance tongue boundaries near non-\tsp s.
If $\{ x_i \}$ denotes the corresponding $\cS$-cycle when it is well-defined,
each boundary corresponds to the collision of one of the points
$x_0$, $x_{-d}$, $x_{(l-1)d}$ and $x_{ld}$ with the \sw.
By Lem.~\ref{le:bc}, at each boundary the corresponding
border-collision matrix, (one of
$P_\cS$, $P_{\cS^{(-d)}}$, $P_{\cS^{((l-1)d)}}$ and $P_{\cS^{(ld)}}$) is singular. 

At a shrinking point there exists an invariant polygon, $\mathcal{P}$.
If $N = 2$, or the shrinking point is terminating,
$\mathcal{P}$ is planar, but in general $\mathcal{P}$ is nonplanar.
$\mathcal{P}$ is comprised of uncountably many
periodic solutions that have a given rotational symbol sequence, $\cS$.
Under variation of parameters, $\mathcal{P}$ may persist as an
invariant topological circle.
In this case there exists curves along which
one intersection point of the circle with the \sw~maps to the other
intersection point in a fixed, finite number of map iterations.
We hypothesize that shrinking points generically occur densely along such
curves, see \cite{SiMe08b}.

The two coexisting periodic solutions in a lens shaped resonance region
have associated rotational symbol sequences
with identical values for $m$ and $n$ but with values of
$l$ that differ by one.
By adding or subtracting one from the value of $l$ of these
symbol sequences, one is able to obtain the
symbol sequences of the periodic solutions in an adjoining lens.
In short, throughout a lens-chain, the rotational number $m/n$ is constant
and $l$ differs by one between lenses.
Numerically we have observed that as $n \to \infty$,
the overall width of the lens-chain tends to zero
and the number of individual lenses increases.
We speculate that in the limit,
a resonance tongue for a quasiperiodic solution of fixed frequency is
a curve along which the fraction of points of the solution that lie in
each half plane changes continuously.

Shrinking points require solutions to intersect the \sw~at two distinct points.
For this reason shrinking points do not occur for
(\ref{eq:pwaMap}) in one dimension because here the \sw~is a single point.
The one-dimensional, piecewise-linear, circle map studied in \cite{YaHa87}
is able to exhibit lens-chain structures because it has
two switching manifolds.

Resonance tongues corresponding to periodic solutions with
non-rotational symbol sequences need not exhibit a lens chain structure \cite{SiMe08b}.
But we believe that the collection of all symbol sequences that
may generically exhibit a lens-chain structure is some class
of which rotational symbol sequences are only the simplest type.
For instance a period-$n$ cycle of (\ref{eq:pwaMap})
may undergo a Neimark-Sacker-like
bifurcation when a complex conjugate pair of associated multipliers
crosses the unit circle.
If this crossing occurs at $\lambda = {\rm e}^{\frac{2 \pi i p}{q}}$,
where $p,q \in \mathbb{Z}$ are coprime,
then a period-$nq$ orbit may arise in
a manner akin to a \tsp.
Numerically we have observed that such an orbit may exhibit a lens-chain shaped
resonance tongue, even though its corresponding symbol sequence 
is non-rotational.
This scenario corresponds to an $n^{\rm th}$ iterate map
of the same form as (\ref{eq:pwaMap})
exhibiting lens-chains.

The analysis presented in this paper applies to periodic solutions
near the border-collision bifurcation of a fixed point.
In general a periodic solution may result from global dynamics and
enter regions of phase space bounded by many distinct \sw s.
A study of the codimension-two points
resulting from such a periodic solution simultaneously colliding
with different \sw s is left for future investigations.

\appendix

\section{Proofs of Lems.~\ref{le:bc} and \ref{le:sn}}
\label{sec:PROOFSa}

\begin{proof}[Proof of Lem.~\ref{le:bc}]
We first show that $(I-A_\sL)$ and $(I-A_\sR)$ cannot both be singular.
Suppose otherwise.
Then $\varrho^{\sf T}(I-A_\sL) = \det(I-A_\sL) e_1^{\sf T} = 0$
and $\varrho^{\sf T}(I-A_\sR) = \det(I-A_\sR) e_1^{\sf T} = 0$.
Thus $\varrho^{\sf T} A_\sL = \varrho^{\sf T} A_\sR = \varrho^{\sf T}$.
Therefore $\varrho^{\sf T} M_\cS = \varrho^{\sf T}$,
i.e.~$M_\cS$ has an eigenvalue, $\lambda = 1$,
which contradicts the assumption that $(I-M_\cS)$ is nonsingular.
Thus at least one of $(I-A_\sL)$ and $(I-A_\sR)$ is nonsingular.
Suppose w.l.o.g.~that $(I-A_\sL)$ is nonsingular.

For any $\hat{b} \in \mathbb{R}^N$,
let $k = \frac{\varrho^{\sf T}\hat{b}}{\varrho^{\sf T}b}$,
let $c = \hat{b} - k b$
and let $y = \mu (I-A_\sL)^{-1} c$.
Notice $\varrho^{\sf T} c = 0$,
thus the first component of $y$ is zero,
hence $y = \mu c + A_\sL y = \mu c + A_\sR y$.
Thus $y = \mu P_\cS c + M_\cS y$ and therefore
$y = \mu (I-M_\cS)^{-1} P_\cS c$.

Let $\hat{x}_0 = k x_0 + y$.
Multiplication on the left by $e_1^{\sf T}$ yields $\hat{s}_0 = k s_0$.
Using above results we have $\hat{x}_0 = \mu (I-M_\cS)^{-1} P_\cS \hat{b}$.
Thus if $\hat{b} \in {\rm null}(P_\cS)$, then $\hat{x}_0 = 0$.
Alternatively if $\hat{b}$ is chosen such that $\varrho^{\sf T} \hat{b} = 0$,
then $\hat{b} = c$ and therefore $\hat{x}_0 = y$.
Hence if $\hat{b} \in {\rm null}(P_\cS)$ and $\varrho^{\sf T} \hat{b} = 0$,
then $y = 0$ and therefore $c = \hat{b} = 0$.

Now suppose $P_\cS$ is singular.
Then there exists $\hat{b} \in {\rm null}(P_\cS)$ with $\hat{b} \ne 0$
and by the previous result, $\varrho^{\sf T} \hat{b} \ne 0$.
Thus we have $\hat{s}_0 = 0$ and $k \ne 0$, therefore $s_0 = 0$,
i.e.~$x_0$ lies on the \sw.

Conversely suppose $s_0 = 0$.
Then $\hat{s}_0 = \mu e_1^{\sf T} (I-M_\cS)^{-1} P_\cS \hat{b} = 0$
for any $\hat{b} \in \mathbb{R}^N$.
Hence $P_\cS$ is singular.
\end{proof}

\begin{proof}[Proof of Lem.~\ref{le:sn}]
Clearly if $(I-M_\cS)$ is nonsingular,
(\ref{eq:nCycleSolutionSystem}) has the unique solution
(\ref{eq:nCycleSolution}).
To prove the converse, suppose now that $(I-M_\cS)$ is singular.
We will show that (\ref{eq:nCycleSolutionSystem}) has no solution.
Let $K = {\rm Span}(\varrho)^{\perp}$
(the collection of all vectors orthogonal to $\varrho$).
Choose any $c \in K$.
We will show that the linear system,
$(I-A_\sL) x = \mu c$,
always has a solution, $x$, with
$s = e_1^{\sf T} x = 0$.
This is clear if $(I-A_\sL)$ is nonsingular for then
$x = \mu (I-A_\sL)^{-1} c$
and $s = \frac{\mu \varrho^{\sf T} c}{\det(I-A_\sL)} = 0$.
Instead, suppose $(I-A_\sL)$ is singular.
We now characterize the nullspace of $(I-A_\sL)$.

We first show that if $v_0 \in {\rm null}(I-A_\sL)$ and $v_0 \ne 0$,
then $e_1^{\sf T} v_0 \ne 0$.
Suppose for a contradiction $e_1^{\sf T} v_0 = 0$.
Let $B_{ij}$ denote the $(N-1) \times (N-1)$ matrix
formed by removing the $i^{\rm th}$ row
and $j^{\rm th}$ column from $(I-A_\sL)$.
Let $\check{v} \in \mathbb{R}^{N-1}$ denote the
last $(N-1)$ elements of $v_0$.
Then $\check{v} \ne 0$ and $B_{i1} \check{v} = 0$ for each $i$,
thus $\det(B_{i1}) = 0$ for each $i$.
In other words the first column of the cofactor matrix of $(I-A_\sL)$ is zero,
equivalently $\varrho^{\sf T} = 0$, which contradicts
the assumption: $\varrho^{\sf T} b \ne 0$.
Hence $e_1^{\sf T} v_0 \ne 0$.

Let $u,v \in {\rm null}(I-A_\sL)$.
Then the linear combination,
$u e_1^{\sf T} v - v e_1^{\sf T} u$,
is an element of the null space of $(I-A_\sL)$
and the first element of this vector is zero.
By the previous argument this vector must be the zero vector,
thus $u$ and $v$ are linearly dependent.
Hence ${\rm null}(I-A_\sL)$ is one-dimensional, therefore
${\rm null}(I-A_\sL) = {\rm Span}(v_0)$.

We have $\varrho^{\sf T} (I-A_\sL) = \det(I-A_\sL) e_1^{\sf T} = 0$,
thus ${\rm null}((I-A_\sL)^{\sf T})
= {\rm Span}(\varrho) = K^\perp$.
By the fundamental theorem of linear algebra,
${\rm range}(I-A_\sL) = {\rm null}((I-A_\sL)^{\sf T})^\perp = K$.
Notice $\mu c \in K$,
therefore there exists $w \in \mathbb{R}^N$ such that
$(I-A_\sL) w = \mu c$.
Let $x = w - \frac{e_1^{\sf T} w}
{e_1^{\sf T} v_0} v_0$.
Then $(I-A_\sL) x = \mu c$ and $s = 0$.
Summarizing, regardless of whether $(I-A_\sL)$ is singular
or nonsingular, we have found a vector, $x$, with zero first component,
such that $(I-A_\sL) x = \mu c$.

The vector $x$ satisfies
$x = \mu c + A_\sL x
= \mu c + A_\sR x$, therefore
$x = \mu P_\cS c + M_\cS x$
and hence $\mu P_\cS c \in {\rm range}(I-M_\cS)$.
But $c$ is arbitrary,
hence $P_\cS K \subset {\rm range}(I-M_\cS)$.
$K$ is $(N-1)$-dimensional thus $P_\cS K = {\rm range}(I-M_\cS)$.
By assumption, $b \ne K$, thus 
(\ref{eq:nCycleSolutionSystem}) has no solution.
\end{proof}

\section{Proof of Lem.~\ref{le:shrinkPoly}}
\label{sec:PROOFSb}

\begin{proof} We begin with the case of non-\tsp s.

\begin{enumerate}[label=\alph{*}),ref=\alpha{*}]
\item First, $P_{\check{\cS}} = P_\cS$ is singular by assumption,
thus by Lem.~\ref{le:bc}, $t = 0$.
Second, $P_{\check{\cS}^{(ld)}} = P_{\cS^{\overline{0}(ld)}}
= P_{\cS^{((l-1)d)\overline{0}}} = P_{\cS^{((l-1)d)}}$
is singular by assumption,
where we have utilized Lem.~\ref{le:oss}\ref{it:switchS}
for the second equality
and (\ref{eq:PindepS0}) for the last equality.
Thus, by Lem.~\ref{le:bc}, $t_{ld} = 0$.
\item Suppose for a contradiction, $t_d = 0$.
Since also $t_{ld} = 0$ from (a),
by a double application of Lem.~\ref{le:solvesAlso},
$p$ solves the \ncss~of
$\check{\cS}^{\overline{d}\,\overline{ld}}$.
From the definition, (\ref{eq:rotSSdef}),
it is seen that
$\check{\cS}^{\overline{d}\,\overline{ld}} =  \check{\cS}^{(-d)}$.
Therefore, $p_d$ solves the \ncss~of $\check{\cS}$, hence $p = p_d$.
The point $p$ is admissible by assumption,
therefore $p = p_{kd}$ for any integer $k$.
Putting $k = m$ we obtain $p = p_1$,
hence $p$ is a fixed point of (\ref{eq:pwaMap})
and lies on the \sw.
However (\ref{eq:pwaMap}) cannot have a fixed point
on the \sw~because by assumption $\mu \ne 0$
and $\varrho^{\sf T} b \ne 0$,
so we have a contradiction.
Thus $t_d \ne 0$.
The remaining three points may be proven nonzero
in a similar fashion.
The given signs follow immediately from admissibility.
\item Suppose for a contradiction,
$\{ p_i \}$ is of period $\tilde{n} < n$.
Clearly $\tilde{n}$ divides $n$.
Then $p_{i \tilde{n} + k} = p_k$
for any integers $i$ and $k$.
But from part (b), $p_d = p_{i \tilde{n} + d}$
lies left of the \sw.
Thus $\check{\cS}_{i \tilde{n} + d} = \sL$ for
$i = 0,\ldots,\frac{n}{\tilde{n}}-1$,
since the orbit is admissible.
But ${\rm gcd}(d,\tilde{n}) = 1$,
hence in view of (\ref{eq:rotSSdef}),
we must have $l \ge 2 + (\frac{n}{\tilde{n}}-1) \tilde{n}
= n - \tilde{n} + 2 \ge \frac{n}{2} + 2$.
Similarly, also from part (b),
$p_{-d} = p_{i \tilde{n} - d}$
lies right of the \sw~and
hence $\check{\cS}_{i \tilde{n} - d} = \sR$
for $i = 0,\ldots,\frac{n}{\tilde{n}}-1$.
Here it follows $l < \tilde{n} \le \frac{n}{2}$
which provides a contradiction.
\item From part (a), $t = 0$, thus
by Lem.~\ref{le:solvesAlso}, $p$ solves the
\ncss~of $\check{\cS}^{\overline{0}} = \cS$.
Similarly, since $t_{ld} = 0$,
$p$ solves the \ncss~of $\check{\cS}^{\overline{ld}} = \cS^{(-d)}$
by (\ref{eq:rotSSdef}) and therefore
$p_d$ solves the \ncss~of $\cS$.
Let $w(\tau) = \tau p + (1-\tau) p_d$.
Since $\{ p_i \}$ is assumed to be admissible,
$\{ w_i(\tau) \}$ is an admissible $\cS$-cycle
whenever $0 \le \tau < 1$.

From part (c), the $n$ points $\{ p_i \}$ are distinct.
We now show that the union of all the $\{ w_i(\tau) \}$, call it $\mathcal{P}$,
has no self-intersections.
Suppose for a contradiction that
$w_{i_1}(\tau_1) = w_{i_2}(\tau_2)$
for some $i_1 \ne i_2$ and $0 < \tau_1, \tau_2 < 1$.
Let $x = w_0(\tau_1)$,
then $\{ x_i \}$ is an admissible $\cS$-cycle
and an admissible $\cS^{(i_2-i_1)}$-cycle.
Since $\tau_1 \ne 0,1$, by part (b),
$s_0 < 0$ and $s_{-d} > 0$.
Therefore $\cS_{i_2-i_1} = \sL$ and
$\cS_{i_2-i_1-d} = \sR$.
By Lem.~\ref{le:oss}\ref{it:permDiffS}, $i_2 - i_1 = 0$, which is a contradiction.
Hence $\mathcal{P}$ has no self-intersections and is
therefore an invariant, nonplanar $n$-gon.

To relate (\ref{eq:pwaMap}) to a map on the unit circle
we describe a bijection between $\mathbb{S}^1$ and $\mathcal{P}$.
The angular coordinate, $\theta \in [0,2\pi)$,
uniquely describes a point on $\mathbb{S}^1$.
Let $j(\theta) = \lfloor \frac{n \theta}{2\pi} \rfloor$
and $\tau(\theta) = \frac{n \theta}{2\pi} - j(\theta)$.
Let $z : \mathbb{S}^1 \to \mathcal{P}$
be defined by $z(\theta) = w_{j(\theta)d}(\tau(\theta))$.
It is easily verified the function $z$ is a bijection.
The induced map on $\mathbb{S}^1$ is
$g(\theta) = (z^{-1} \circ f \circ z)(\theta)
= \theta + \frac{2\pi m}{n}$,
i.e.~rigid rotation with rotation number $m/n$.
\end{enumerate}
\end{proof}

\begin{proof} We now consider the case of  \tsp s.

Let $v = y + {\rm i} z$ be an eigenvector corresponding to the multiplier
$\lambda = {\rm e}^{\pm \frac{2 \pi {\rm i} m}{n}}$ for the matrix $A_\sL$.
Let $D = \left[ \begin{array}{cc}
\cos(\frac{2 \pi m}{n}) & \sin(\frac{2 \pi m}{n}) \\
-\sin(\frac{2 \pi m}{n}) & \cos(\frac{2 \pi m}{n})
\end{array} \right]$.
Then $D^j = \left[ \begin{array}{cc}
\cos(\frac{2 \pi j m}{n}) & \sin(\frac{2 \pi j m}{n}) \\
-\sin(\frac{2 \pi j m}{n}) & \cos(\frac{2 \pi j m}{n})
\end{array} \right]$
and $A_\sL^j [y~z] = [y~z] D^j$.
Notice also $D^n = I$.
The first component of $v$ must be nonzero because otherwise
$A_\sR [y~z] = [y~z] D$ and hence
$M_{\check{\cS}} [y~z] = [y~z] D^n = [y~z]$,
violating the assumption that $(I-M_{\check{\cS}})$ is nonsingular.
Thus $E_c$ intersects the \sw.
Furthermore, $E_c$ is two-dimensional, because otherwise
there exists another eigenvector for
$\lambda = {\rm e}^{\pm \frac{2 \pi {\rm i} m}{n}}$,
call it $\hat{v}$,
linearly independent to $v$,
and the linear combination,
$v e_1^{\sf T} \hat{v} - \hat{v} e_1^{\sf T} v$,
is an eigenvector with zero first component contradicting the
previous argument.

Let $x = x_0 = \alpha y + \beta z + x^{*(\sL)}$
for some as yet undetermined scalars $\alpha$ and $\beta$.
Let $x_i$ denote the $i^{\rm th}$ iterate of $x$
via the \sew~$\hat{\cS} = \sL^n$.
Then $x_i = [y~z] D^i
[\alpha~\beta]^{\sf T}
+ x^{*(\sL)}$
and hence $\{ x_i \}$ is a period-$n$ cycle.
We wish to choose $\alpha$ and $\beta$ such that
$s_0 = e_1^{\sf T} [y~z] 
[\alpha~\beta]^{\sf T}
+ s^{*(\sL)} = 0$ and
$s_{-d} = e_1^{\sf T} [y~z] D^{-d}
[\alpha~\beta]^{\sf T}
+ s^{*(\sL)} = 0$.
Combining these two equations yields the linear system,
$X
[\alpha~\beta]^{\sf T}
= -s^{*(\sL)}
[1~1]^{\sf T}
$
where
$X = \left[ \begin{array}{cc}
e_1^{\sf T} y & e_1^{\sf T} z \\
e_1^{\sf T} y \cos(\frac{2 \pi}{n}) + e_1^{\sf T} z \sin(\frac{2 \pi}{n}) &
-e_1^{\sf T} y \sin(\frac{2 \pi}{n}) + e_1^{\sf T} z \cos(\frac{2 \pi}{n})
\end{array} \right]$
has determinant,
$\det(X) = -\left( (e_1^{\sf T} y)^2 + (e_1^{\sf T} z)^2 \right) \sin(\frac{2 \pi}{n})
\ne 0$.
Thus we let 
$
[\alpha~\beta]^{\sf T}
= -s^{*(\sL)} X^{-1} 
[1~1]^{\sf T}
$.
It follows that
\begin{equation}
s_{id} = s^{*(\sL)} \left(
1 - \frac{\cos(\frac{2 \pi (i + \frac{1}{2})}{n})}{\cos(\frac{\pi}{n})} \right) \;.
\label{eq:sid}
\end{equation}

Thus the period-$n$ cycle, $\{ x_i \}$, is admissible
and solves the \ncss~of $\check{\cS}$.
Therefore $x = p$.
Thus $\{ p_i \}$ has period $n$ which verifies part (c) of the lemma.
Also $t_i = s_i$, for each $i$, thus by (\ref{eq:sid})
we have parts (a) and (b).
To verify part (d), note that
as in the non-terminating case,
$p$ and $p_d$ are admissible solutions to the \ncss~of $\cS$.
Thus $w(\tau) = \tau p + (1 - \tau) p_d$ is an admissible solution
to the \ncss~of $\cS$ for all $\tau \in [0,1)$.
The union of all such cycles
is an invariant, $\mathcal{P}$,
which lies on $E_c$, so is planar.
The points, $p_i$, lie on an ellipse and ordered so that $\mathcal{P}$ has no
self-intersections and is an $n$-gon.
As before, a bijection between $\mathbb{S}^1$ and $\mathcal{P}$ may be constructed
to show that the restriction of (\ref{eq:pwaMap}) to $\mathcal{P}$
is homeomorphic to rigid rotation with rotation number $m/n$.
\end{proof}


\end{document}